\documentclass[11pt]{amsart}
\usepackage{amsmath,amssymb}
\usepackage{verbatim}
\usepackage{color}
\usepackage{hyperref}

\begin{document}

\newtheorem{thm}{Theorem}[section]
\newtheorem{theorem}{Theorem}[section]
\newtheorem{lem}[thm]{Lemma}
\newtheorem{lemma}[thm]{Lemma}
\newtheorem{prop}[thm]{Proposition}
\newtheorem{proposition}[thm]{Proposition}
\newtheorem{corollary}[thm]{Corollary}
\newtheorem{definition}[thm]{Definition}
\newtheorem{remark}[thm]{Remark}
\newtheorem{conjecture}[theorem]{Conjecture}

\numberwithin{equation}{section}

\newcommand{\Z}{{\mathbb Z}} 
\newcommand{\Q}{{\mathbb Q}}
\newcommand{\R}{{\mathbb R}}
\newcommand{\C}{{\mathbb C}}
\newcommand{\N}{{\mathbb N}}
\newcommand{\FF}{{\mathbb F}}
\newcommand{\fq}{\mathbb{F}_q}
\newcommand{\rmk}[1]{\footnote{{\bf Comment:} #1}}

\newcommand{\bfA}{{\boldsymbol{A}}}
\newcommand{\bfY}{{\boldsymbol{Y}}}
\newcommand{\bfX}{{\boldsymbol{X}}}
\newcommand{\bfZ}{{\boldsymbol{Z}}}
\newcommand{\bfa}{{\boldsymbol{a}}}
\newcommand{\bfy}{{\boldsymbol{y}}}
\newcommand{\bfx}{{\boldsymbol{x}}}
\newcommand{\bfz}{{\boldsymbol{z}}}
\newcommand{\F}{\mathcal{F}}
\newcommand{\Gal}{\mathrm{Gal}}
\newcommand{\Fr}{\mathrm{Fr}}
\newcommand{\Hom}{\mathrm{Hom}}
\newcommand{\GL}{\mathrm{GL}}

\renewcommand{\mod}{\;\operatorname{mod}}
\newcommand{\ord}{\operatorname{ord}}
\newcommand{\TT}{\mathbb{T}}
\renewcommand{\i}{{\mathrm{i}}}
\renewcommand{\d}{{\mathrm{d}}}
\renewcommand{\^}{\widehat}
\newcommand{\HH}{\mathbb H}
\newcommand{\Vol}{\operatorname{vol}}
\newcommand{\area}{\operatorname{area}}
\newcommand{\tr}{\operatorname{tr}}
\newcommand{\norm}{\mathcal N} 
\newcommand{\intinf}{\int_{-\infty}^\infty}
\newcommand{\ave}[1]{\left\langle#1\right\rangle} 
\newcommand{\Var}{\operatorname{Var}}
\newcommand{\Prob}{\operatorname{Prob}}
\newcommand{\sym}{\operatorname{Sym}}
\newcommand{\disc}{\operatorname{disc}}
\newcommand{\CA}{{\mathcal C}_A}
\newcommand{\cond}{\operatorname{cond}} 
\newcommand{\lcm}{\operatorname{lcm}}
\newcommand{\Kl}{\operatorname{Kl}} 
\newcommand{\leg}[2]{\left( \frac{#1}{#2} \right)}  
\newcommand{\Li}{\operatorname{Li}}

\newcommand{\sumstar}{\sideset \and^{*} \to \sum}

\newcommand{\LL}{\mathcal L} 
\newcommand{\sumf}{\sum^\flat}
\newcommand{\Hgev}{\mathcal H_{2g+2,q}}
\newcommand{\USp}{\operatorname{USp}}
\newcommand{\conv}{*}
\newcommand{\dist} {\operatorname{dist}}
\newcommand{\CF}{c_0} 
\newcommand{\kerp}{\mathcal K}

\newcommand{\Cov}{\operatorname{cov}}
\newcommand{\Sym}{\operatorname{Sym}}

\newcommand{\ES}{\mathcal S} 
\newcommand{\EN}{\mathcal N} 
\newcommand{\EM}{\mathcal M} 
\newcommand{\Sc}{\operatorname{Sc}} 
\newcommand{\Ht}{\operatorname{Ht}}

\newcommand{\E}{\operatorname{E}} 
\newcommand{\sign}{\operatorname{sign}} 

\newcommand{\divid}{d} 

\newcommand{\h}{\mathcal{H}_{2g+1,q}}
\newcommand{\p}{\mathcal{P}_{2g+1,q}}
\newcommand{\f}{\mathbb{F}_{q}[T]}
\newcommand{\z}{\zeta_A}
\newcommand{\lo}{\log_q}
\newcommand{\x}{\chi}
\newcommand{\xx}{\mathcal{X}}
\newcommand{\lL}{\mathcal{L}}
\newcommand{\e}{\varepsilon}

\title[THE INTEGRAL MOMENTS AND RATIOS CONJECTURES]
{The Integral Moments and Ratios of Quadratic Dirichlet $L$-Functions over Monic Irreducible Polynomials in $\mathbb{F}_{q}[T]$}

\author{Julio Andrade}
\address{Department of Mathematics, University of Exeter, Exeter, EX4 4QF, United Kingdom}
\email{j.c.andrade@exeter.ac.uk}

\author{Hwanyup Jung}
\address{Department of Mathematics Education, Chungbuk National University, Cheongju 361-763, Korea}
\email{hyjung@chungbuk.ac.kr}

\author{ASMAA sHAMESALDEEN}
\address{Department of Mathematics, University of Exeter, Exeter, EX4 4QF, United Kingdom}
\email{as1029@exeter.ac.uk}


\subjclass[2010]{Primary 11M38; Secondary 11M06, 11G20, 11M50, 14G10}
\keywords{function fields, integral moments of $L$--functions, quadratic Dirichlet $L$--functions, ratios conjecture}

\begin{abstract}
In this paper we extend to the function field setting the heuristics formerly developed by Conrey, Farmer, Keating, Rubinstein and Snaith, for the integral moments of $L$-functions. We also adapt to the function setting the heuristics first developed by Conrey, Farmer and Zirnbauer to the study of mean values of ratios of $L$-functions. Specifically, the focus of this paper is on the family of quadratic Dirichlet $L$-functions $L(s,\chi_{P})$ where the character $\x$ is defined by the Legendre symbol for polynomials in $\mathbb{F}_{q}[T]$ with $\mathbb{F}_{q}$ a finite field of odd cardinality and the averages are taken over all monic and irreducible polynomials $P$ of a given odd degree. As an application we also compute the formula for the one-level density for the zeros of these $L$-functions. 
\end{abstract}
\date{\today}

\maketitle

\section{Introduction}

A central topic in analytic number theory is the study of moments of families of $L$-functions. Many fine mathematicians have studied this subject and considerable progress was made in the last decades in the direction of getting a better understanding of the asymptotic behaviour of such moments. For example, in the case of the Riemann zeta function, the problem is to understand the asymptotic behaviour of  

\begin{equation}
M_k(T)=\int_{0}^{T} \left|\zeta\left(\tfrac{1}{2}+it\right)\right|^{2k}dt,
\end{equation}
as $T\to\infty$. 
\newline

Hardy and Littlewood \cite{hardy-littlewood} proved in 1918 an asymptotic formula for the second moment, i.e., 

\begin{equation}
M_1(T)\sim T\log T.
\end{equation}
\newline

In 1926 Ingham \cite{ingham} showed that when $k=2$,

\begin{equation}\label{mk}
M_2(T)\sim \frac{1}{2\pi^2} T \left(\log T\right)^4.
\end{equation}
\newline

For values of $k\geq3$ it still remains an unsolved problem to obtain asymptotic formulas for $M_{k}(T)$, however, it is conjectured that for every $k \ge 0$ there is a constant $c_k$ such that

\begin{equation}
M_k(T)\sim c_k T \left(\log T\right)^{k^2}.
\end{equation} 
\newline

Conrey and Ghosh \cite{CG} made a conjecture for the sixth moment of the Riemann zeta-function and later on Conrey and Gonek \cite{C-Go} put forward a conjecture for the eight-moment but they approach fails to provide conjectures for higher moments. Keating and Snaith \cite{keating snaith2}, using random matrix theory, conjectured the precise value of the constant $c_k$ for all values of $k$ with $\mathfrak{R}(k)>1/2$. More recently Conrey and Keating, in a series of paper \cite{CK1, CK2, CK3, CK4} returned to the problem of obtaining conjectures for the higher moments of the Riemann zeta-function using only number-theoretic heuristics. Their new approach not only produce the conjectures for the moments of the Riemann zeta-function as well as explain the role of non-diagonal contribution to the main terms in the asymptotic formulas. 
\newline

A different example is the family of quadratic Dirichlet $L$-functions $L(s,\x_d)$, where $\x_d$ is the real primitive Dirichlet character modulo $d$ defined by the Kronecker symbol $\x_d(n)=\left(\frac{d}{n}\right)$. The problem here is to establish an asymptotic formula for 

\begin{equation}
\sum_{d\le X} L\left(\tfrac{1}{2},\x_d\right)^k,
\end{equation}  
as $X\to \infty$, where the sum is taken over all positive discriminants $d$ and $k$ is a positive integer. In this case, as it is for the Riemann zeta-function, just the first few moments were computed. In 1981 Jutila \cite{jutila} established the asymptotic formula for the first and second moments. The asymptotic formulas he obtained are

\begin{equation}\label{k=1}
\sum_{d\le X} L\left(\tfrac{1}{2},\x_d\right)\sim C_1 X \log X,
\end{equation}
and 

\begin{equation}\label{k=2}
\sum_{d\le X} L\left(\tfrac{1}{2},\x_d\right)^2\sim C_2 X \left(\log X\right)^3,
\end{equation}
where the constants $C_1$ and $C_2$ can be expressed in terms of Euler products and factors containing the Riemann zeta function. Soundararajan \cite{soundararajan} computed the asymptotic formula for the third moment, he proved that

\begin{equation}\label{k=3}
\sum_{d\le X} L\left(\tfrac{1}{2},\x_{8d} \right)^3\sim C_3 X \left(\log X\right)^6,
\end{equation} 
where $d$ is an odd, square-free and positive number, $\x_{8d}$ is real, even primitive Dirichlet character with conductor $8d$, and $C_3$ is a constant. 
\newline

In general, it is conjectured that 

\begin{equation}
\sum_{d\le X} L\left(\tfrac{1}{2},\x_{d} \right)^k\sim C_k X \left(\log X\right)^{\frac{k(k+1)}{2}}.
\end{equation} 
As before, making use of random matrix theory, Keating and Snaith \cite{keating snaith} conjectured in their paper the precise value of $C_k$.
\newline

In 2005 Conrey, Farmer, Keating, Rubinstein and Snaith \cite{CFKRS} presented a new heuristic for all of the main terms in the integral moments of several families of primitive $L$-functions. Their conjectures agrees with previous known results. For the Riemann zeta function, they gave a precise conjecture for $M_k(T)$ including an asymptotic expansion for the lower order terms using shifted moments. For the family of quadratic Dirichlet $L$-functions their conjecture reads.

%

\begin{conjecture}(Conrey, Farmer, Keating, Rubinstein, Snaith)\label{CFKRSconjecture}
	Let $X_d(s)=|d|^{\frac{1}{2}-s}X(s,a)$ where $a=0$ if $d>0$ and $a=1$ if $d<0,$ and 
	
	\begin{equation}
	X(s,a)=\pi^{s-\frac{1}{2}} \Gamma\left(\frac{1+a-s}{2}\right)\Big/ \Gamma\left(\frac{s+a}{2}\right).
	\end{equation}  
	That is, $X_d(s)$ is the factor in the functional equation for the quadratic Dirichlet $L$-function
	
	\begin{equation}
	L(s,\x_d)=\varepsilon_d X_d L(1-s,\x_d).
	\end{equation} 
	Summing over fundamental discriminants $d$, we have 
	
	\begin{equation}\label{CFKRS moment conjecture equation equation}
	\sum_{d} L(\tfrac{1}{2},\x_d)^k=\sum_{d} Q_k (\log|d|)(1+o(1)),
	\end{equation}
	where $Q_k$ is polynomial of degree $k(k+1)/2$ given by the $k$-fold residue 
	
	\begin{equation}
	\begin{split}
	Q_k(x) = & \frac{(-1)^{k(k-1)/2} \, 2^k}{k!} \, \frac{1}{(2\pi i )^k} \, \oint \cdots \oint \frac{G(z_1, \cdots, z_k) \triangle (z_1^2, \cdots, z_k^2)^2}{\prod_{i=1}^k z_i^{2k-1}}\\
	& \times  e^{\frac{x}{2}\sum_{i=1}^k z_i} \, dz_1 \cdots z_k,
	\end{split}
	\end{equation}
	with
	
		\begin{equation}
	G(z_1, \cdots, z_k)= A_k(z_1, \cdots, z_k) \prod_{i=1}^k \xx(\tfrac{1}{2}+z_i)^{-\frac{1}{2}} \prod_{1 \leq i \leq j \leq k} \zeta(1+z_i+z_j),
	\end{equation}
	$\Delta (z_1, \cdots, z_k)$ the Vandermonde determinant given by 
	
	\begin{equation}\label{vandermonde}
	\Delta (z_1, \cdots, z_k)=\prod_{1\leq i < j \leq k} (z_j - z_i),
	\end{equation}
	and $A_k$ is the Euler product, absolutely convergent for $|\Re(z_i)|<\frac{1}{2},$ defined by 
	
	\begin{equation}
	\begin{split}
	A_k(z_1, \cdots, z_k)=& \prod_{p} \prod_{1 \leq i \leq j \leq k} \Bigg(1- \frac{1}{p^{1+z_i+z_j}}\Bigg)\\
	& \times \frac{1}{2} \left( \left(\prod_{i=1}^k \Bigg(1-\frac{1}{p^{\frac{1}{2}+z_i}}\Bigg)^{-1}+ \prod_{i=1}^k \Bigg(1+\frac{1}{p^{\frac{1}{2}+z_i}}\Bigg)^{-1}\right)+\frac{1}{p} \right)\\
	& \times \left(1+\frac{1}{p}\right)^{-1}.
	\end{split}
	\end{equation}
\end{conjecture}

It is important to observe that Diaconu, Goldfeld and Hoffstein \cite{DGH} have also conjectured moments of families of $L$-functions using different techniques, their method is based on multiple Dirichlet series. Recently, Diaconu and Whitehead \cite{DW} established a smoothed asymptotic formula for the third moment of quadratic Dirichlet $L$-functions at the central value. In addition to the main term, which is known, they prove the existence of a secondary term of size $x^{3/4}$. The error term in their asymptotic formula is on the order of $O(x^{2/3+\delta})$ for every $\delta>0$.
\newline

Conrey, Farmer and Zirnbauer \cite{Conr-Far-Zir} presented a generalisation of the heuristic method for moments presented in \cite{CFKRS} to the case of ratios of product of $L$-functions. These conjectures are very powerful since they encode information about statistics of zeros of such $L$-functions. The ratios conjectures as put forward by Conrey, Farmer and Zirnbauer can be used to prove very precise conjectures about distribution of zeros of families of $L$-functions such as pair-correlation and $n$-level density (for more details see \cite{con-snaith appl int mom}). Their ratios conjecture for the family of quadratic Dirichlet $L$-functions is as follow.

\begin{conjecture}\label{CFZ ratios conjecture} (Conrey, Farmer, Zirnbauer)
	Let $\mathcal{D}^+=\{L(s,\x_d:d>0)\}$ to be the symplectic family of $L$-functions associated with the quadratic character $\x_d$. For positive real parts of $\alpha_k$ and $\gamma_m$ we have
	
	\begin{equation}\label{CFZ ratios conjecture equation}
\begin{split}
	&\sum_{0<d\le X} \frac{\prod_{k=1}^K L(\frac{1}{2}+\alpha_k,\x_d)}{\prod_{m=1}^Q L(\frac{1}{2}+\gamma_m,\x_d)}\\
	& = \sum_{0<d\le X} \sum_{\varepsilon\in\{-1,1\}^K}\left(\frac{|d|}{\pi}\right)^{\frac{1}{2}\sum_{k=1}^K \left(\varepsilon_k\alpha_k-\alpha_k\right)}\\
	& \times \prod_{k=1}^K g_+\left(\tfrac{1}{2}+\tfrac{\alpha_k-\varepsilon_k\alpha_k}{2}\right) Y_S A_\mathcal{D}\left(\varepsilon_1\alpha_1,\cdots,\varepsilon_K\alpha_K;\gamma\right) + o(X),
\end{split}
	\end{equation}
	where
	
	\begin{equation}
	g_+(s)= \frac{\Gamma\left(\frac{1-s}{2}\right)}{\Gamma\left(\frac{s}{2}\right)},
	\end{equation}
	\begin{equation}
	\begin{split}
	Y(\alpha;\gamma)= \frac{\prod_{j\le k\le K} \zeta\left(1+\alpha_j+\alpha_k\right) \prod_{m\le r\le Q} \zeta\left(1+\gamma_m+\gamma_r\right)}{\prod_{k=1}^K\prod_{m=1}^Q \zeta\left(1+\alpha_k+\gamma_m\right)}.
	\end{split}
	\end{equation}
	and
	
	\begin{equation}
\begin{split}
A_\mathcal{D}&(\alpha;\gamma)=\prod_{\substack{p}} \frac{\prod_{j\le k\le K} \left(1-\frac{1}{p^{1+\alpha_j+\alpha_k}}\right) \prod_{m\le r\le Q} \left(1-\frac{1}{p^{1+\gamma_m+\gamma_r}}\right)}{\prod_{k=1}^K\prod_{m=1}^Q \left(1-\frac{1}{p^{1+\alpha_k+\gamma_m}}\right)}\\
& \times \left(1+ \left(1+\frac{1}{p}\right)^{-1} \sum_{0< \sum_ka_k+\sum_mc_m \text{ is even}} \frac{\prod_{m=1}^Q \mu\left(P^{c_m}\right)}{p^{\sum_ka_k(\frac{1}{2}+\alpha_k)+\sum_m c_m(\frac{1}{2}+\gamma_m)}}\right).
\end{split}
\end{equation}
\end{conjecture}

In 1979 Goldfeld and Viola \cite{goldfeld viola} introduced a variant of the problem about moments of quadratic Dirichlet $L$-functions. They conjectured an asymptotic formula for

\begin{equation}
\sum_{\substack{p\le X \\ p \equiv 3 \text{ (mod} 4)}} L\left(\tfrac{1}{2},\x_p\right),
\end{equation}
where the sum is taken over prime numbers and $\x_p(n)= \left(\tfrac{n}{p}\right)$ is the usual Legendre symbol. In this direction, Jutila \cite{jutila} proved that

\begin{equation}\label{prime1}
\sum_{\substack{p\le X \\ p \equiv 3 \text{ (mod} 4)}} \left(\log p\right) L\left(\tfrac{1}{2},\x_p\right) \sim \frac{1}{4} X \log X.
\end{equation}
However, establishing an asymptotic formula for 

\begin{equation}
\sum_{\substack{p\le X \\ p \equiv 3 \text{ (mod} 4)}} L\left(\tfrac{1}{2},\x_p\right)^k,
\end{equation}
when $X\to\infty$ and $k>1,$ still remains an unsolved problem.
\newline


In this paper we consider moments of the symplectic family of quadratic Dirichlet $L$-functions in the function field setting. Similar to the number field case, the goal is to determine the asymptotic behaviour of 

\begin{equation}
\sum_{D\in \h} L\left(\tfrac{1}{2},\x_D\right)^k,
\end{equation}
as $|D|\to\infty$, where $|D|=q^{\text{deg}(D)}$ denotes the norm of the polynomial $D$ and  $L(s,\chi_{D})$ is the quadratic Dirichlet $L$-function associated to the quadratic character $\chi_{D}$ in $\mathbb{F}_{q}[T]$ with $q$ an odd prime power and $\mathbb{F}_{q}$ being the ground field. $\h$ is the hyperelliptic ensemble of monic, square-free polynomials of degree $2g+1$ with coefficients in $\mathbb{F}_{q}$. 
\newline

Andrade and Keating \cite{a&kmeanvalue} and Hoffstein and Rosen \cite{HR} computed the first moment of this family, they showed that

\begin{equation}
\sum_{D\in\h} L\left(\tfrac{1}{2},\x_D\right) \sim |D|P_1\left(\lo|D|\right),
\end{equation}
where $P_1$ is a linear polynomial. For the second, third and fourth moments of this family, Florea \cite{Florea, Florea1} proved that

\begin{equation}
\sum_{D\in\h} L\left(\tfrac{1}{2},\x_D\right)^2 \sim |D| P_2\left(\lo|D|\right),
\end{equation}

\begin{equation}
\sum_{D\in\h} L\left(\tfrac{1}{2},\x_D\right)^3 \sim |D| P_3\left(\lo|D|\right),
\end{equation}
and

\begin{equation}
\sum_{D\in\h} L\left(\tfrac{1}{2},\x_D\right)^4 \sim |D| P_4\left(\lo|D|\right),
\end{equation}
where $P_2$, $P_3$ and $P_4$ are polynomials of degree $3$, $6$ and $10$ respectively whose coefficients can be computed explicitly, except for $P_4$ where only the first few coefficients were obtained. It is also worth to notice that Florea in \cite{Florea2} has improved the error term in the first moment of quadratic Dirichlet $L$-functions in function fields and was able to obtain an strenuous lower order term that was never predicted by random matrix theory and other heuristics in the number field case.
\newline

In another paper, Andrade and Keating \cite{a&kConInMo} adapted the recipe of \cite{CFKRS} and of \cite{Conr-Far-Zir} to the function field setting and conjectured asymptotic formulas for the integral moments and ratios of the family of quadratic Dirichlet $L$-functions in function fields. Their main conjectures are presented below.  

\begin{conjecture}
Let $k$ be a positive integer. Then,

\begin{equation}\label{adrade and Keating moment conjecture}
\sum_{D \in \h} L(\tfrac{1}{2},\x_D)^k=\sum_{D\in \h} Q_k (\lo|D|)(1+o(1))
\end{equation}
where $Q_k$ is polynomial of degree $k(k+1)/2$, with explicit coefficients.
\end{conjecture} 

\begin{conjecture} Let $\alpha_{k}$ and $\gamma_{m}$ complex numbers with positive and small real parts. Then,
\begin{equation}\label{adrade and Keating ratios conjecture}
\begin{split}
\sum_{D\in\h} & \frac{\prod_{k=1}^K L\left(\tfrac{1}{2}+\alpha_k,\x_D\right)}{\prod_{m=1}^Q L\left(\tfrac{1}{2}+\gamma_m,\x_D\right)}\\
& = \sum_{D\in\h} \sum_{\varepsilon\in\{-1,1\}^k} \left|D\right|^{-\frac{1}{2} \sum_{k=1}^K \left(\varepsilon_k\alpha_k-\alpha_k\right)} \prod_{k=1}^K X\left(\tfrac{1}{2}+\tfrac{\alpha_k-\varepsilon_k\alpha_k}{2}\right) \\
& \times Y\left(\varepsilon_1\alpha_1,\cdots,\varepsilon_K\alpha_K\right) A_\mathfrak{D}\left(\varepsilon_1\alpha_1,\cdots,\varepsilon_K\alpha_K\right)+ o\left(D\right),
\end{split}
\end{equation}
with

\begin{equation}
\begin{split}
A_\mathfrak{D}&(\alpha;\gamma)=\prod_{\substack{P \text{ monic} \\ \text{irreducible}}} \frac{\prod_{j\le k\le K} \left(1-\frac{1}{|P|^{1+\alpha_j+\alpha_k}}\right) \prod_{m\le r\le Q} \left(1-\frac{1}{|P|^{1+\gamma_m+\gamma_r}}\right)}{\prod_{k=1}^K\prod_{m=1}^Q \left(1-\frac{1}{|P|^{1+\alpha_k+\gamma_m}}\right)}\\
& \times \left(1+ \left(1+\frac{1}{|P|}\right)^{-1} \sum_{0< \sum_ka_k+\sum_mc_m \text{ is even}} \frac{\prod_{m=1}^Q \mu\left(P^{c_m}\right)}{|P|^{\sum_ka_k(\frac{1}{2}+\alpha_k)+\sum_m c_m(\frac{1}{2}+\gamma_m)}}\right)
\end{split}
\end{equation}
and 

\begin{equation}
\begin{split}
Y(\alpha;\gamma)= \frac{\prod_{j\le k\le K} \z\left(1+\alpha_j+\alpha_k\right) \prod_{m\le r\le Q} \z\left(1+\gamma_m+\gamma_r\right)}{\prod_{k=1}^K\prod_{m=1}^Q \z\left(1+\alpha_k+\gamma_m\right)},
\end{split}
\end{equation}
where $\zeta_{A}(s)$ is the zeta function associated to the polynomial ring $A=\mathbb{F}_{q}[T]$ and $X(s)$ is a function that depends on $q$.
\end{conjecture}

One can note that (\ref{adrade and Keating moment conjecture}) and (\ref{adrade and Keating ratios conjecture}) are the function field analogues of the formulas (\ref{CFKRS moment conjecture equation equation}) and (\ref{CFZ ratios conjecture equation}) respectively. 
\newline

The main aim of this paper is to formulate a conjectural asymptotic formula for  

\begin{equation}\label{primk}
\sum_{P\in\p} L\left(\tfrac{1}{2},\x_P\right)^k,
\end{equation}
where $\p$ is the set of all monic, irreducible polynomials of odd degree $2g+1$ with coefficients in $\mathbb{F}_{q}$, as $|P| \to \infty$. 
\newline

In the paper \cite{a&kprimemean}, Andrade and Keating established asymptotic formulas for the first and second moments of (\ref{primk}), namely 

\begin{equation}\label{andfirst}
\sum_{P\in\p} \left(\lo|P|\right) L\left(\tfrac{1}{2},\x_P\right) \sim \frac{1}{2} |P| \left(\lo|P|+1\right),
\end{equation}

\begin{equation}\label{andsecond}
\sum_{P\in\p}  L\left(\tfrac{1}{2},\x_P\right)^{2} \sim \frac{1}{24}\frac{1}{\z(2)} |P| \left(\lo|P|\right)^2.
\end{equation}
\newline

In this paper we adapt to the function field case the conjectures \ref{CFKRSconjecture} and \ref{CFZ ratios conjecture} for the family of quadratic Dirichlet $L$-functions associated with $\x_P$ over a fixed finite field $\mathbb{F}_q$. In Section \ref{backgrownd}, we present some basic facts on $L$-function over function fields followed by the statement of our main results. In section \ref{Integral moments}, we present the details of the recipe in \cite{CFKRS} when it is adapted for the function field setting. In Section \ref{conjecture formula}, we use the integral moments conjecture over function fields when $k=1,2$, and compare with the main theorems of \cite{a&kConInMo}, then we conjecture the precise value for the third moment, i.e., when $k=3$ in this setting. In Section \ref{ratios conjecture}, we present the recipe of \cite{Conr-Far-Zir} for the same family of $L$-functions over function fields. In Section \ref{one level}, we use the ratios conjecture for function fields and compute the one-level density of the zeros of this same family of $L$-functions.

\section{Statement of the main results}\label{backgrownd}
In this section we gather some basic facts about $L$-functions over function fields. Many of the results and notation here can also be found in \cite{Rosen}.
\newline

Let $\mathbb{F}_q$ be a finite field of odd cardinality $q=p^a$, with $p$ a prime. Denote the polynomial ring over $\mathbb{F}_q$ by $A=\f,$ and the rational function field by $k=\mathbb{F}_q(T).$  For a polynomial $f$ in $\f$ we define the norm of $f$ by $|f|:=q^{\text{deg}(f)}.$ For $\mathfrak{R}(s)>1$, the zeta function attached to $A$ is defined by 

\begin{equation}
\begin{split}
\z(s) &=\sum_{f \text{ monic}} \frac{1}{|f|^s}\\
& = \prod_{\substack{P \text{ monic} \\ \text{irreducible}}} \left(1-|P|^{-s}\right)^{-1}.
\end{split}
\end{equation} 
Since there are $q^n$ monic polynomials of degree $n$, we can easily prove that 

\begin{equation}
\z(s)=\frac{1}{1-q^{1-s}},
\end{equation}
which provides an analytic continuation of the zeta-function to the whole complex plane, with simple pole at $s=1,$ which leads to the analogue of the Prime Number Theorem for polynomials in $A=\f$.

\begin{theorem}(Prime Polynomial Theorem) \label{PNT}
	Let  $\pi_A(n)$ denote the number of monic irreducible polynomials of degree $n$ in $A$. Then 
	
	\begin{equation}
	\pi_A(n)=\frac{q^n}{n} + O\left(\frac{q^{n/2}}{n}\right).
	\end{equation}
\end{theorem} 

Now, Let $P$ be a monic irreducible polynomial, define the quadratic character $\left(\frac{f}{P}\right)$ by 

\begin{equation}
\left(\frac{f}{P}\right)= \begin{cases}
1 & \text{ if } f \text{ is a square (mod }P),P\nmid f\\
-1 & \text{ if } f \text{ is not a square (mod }P),P\nmid f\\
0  & \text{ if } P\mid f.\\
\end{cases}
\end{equation}
The quadratic reciprocity law states that for $A,B$ non-zeros and relatively prime monic polynomials, we have 

\begin{equation}
\left(\frac{A}{B}\right)= \left(\frac{B}{A}\right) \left(-1\right)^{\frac{q-1}{2} \text{deg}(A)\text{deg}(B)}. 
\end{equation}
We denote by $\x_P$ the quadratic character defined in terms of the quadratic residue symbol for $A$

\begin{equation}
\x_P(f)=\left(\frac{P}{f}\right),
\end{equation}
where $f\in A$. 
\newline

In this paper, the focus will be in the family of quadratic Dirichlet $L$-functions associated to polynomials $P\in\p$, where

\begin{equation}
\p=\{P\in A, \text{ monic, irreducible and deg}(P)=2g+1\}. 
\end{equation}
\newline

The quadratic Dirichlet $L$-function attached to the character $\x_P$ is defined to be 

\begin{equation}
\begin{split}
L\left(s,\x_P\right) &= \sum_{\substack{f\in A \\ f \text{ monic}}} \frac{\x_P(f)}{|f|^s} \\
& = \prod_{\substack{P \text{ monic} \\ \text{irreducible}}} \left(1-\x_P(P)|P|^{-s}\right)^{-1},\text{\color{white} fjdhkjd} \mathfrak{R}(s)>1.
\end{split} 
\end{equation}
With the change of variables $u=q^{-s},$ $L(s,\x_P)$ is a polynomial of degree $2g$ given by  
 
\begin{equation}
\begin{split}
L\left(s,\x_P\right) &= \mathcal{L}(u,\chi_{P})= \sum_{n=0}^{2g}  \sum_{\substack{f \text{ monic} \\ \text{deg}(f)=n}}\x_P(f)u^n.
\end{split} 
\end{equation}
(see Propositions 14.6 and 17.7 in \cite{Rosen}).

We are now in a position to state the main conjectures of this paper. 

\begin{conjecture}.\label{ourconjecture}
	Suppose that $q \equiv 1 (\text{mod }4)$ is the fixed cardinality of the finite field $\mathbb{F}_q$ and let $\xx_P(s)=|P|^{1/2-s}\xx(s)$ where 
	$$\xx(s)=q^{-1/2+s}.$$
	That is $\xx_P(s)$ is the factor in the functional equation 
	\begin{equation}\label{functional equation}
	L(s,\x_P)=\xx_P(s)L(1-s,\x_p).
	\end{equation}
	Summing over primes  $P \in \p$ we have 
	
	\begin{equation}
	\sum_{P \in \p} L\left(\tfrac{1}{2},\x_P\right)^k=\sum_{P\in \p} Q_k (\lo|P|)(1+o(1))
	\end{equation}
	where $Q_k$ is polynomial of degree $k(k+1)/2$ given by the $k$-fold residue 
	
	\begin{equation}
	\begin{split}
	Q_k(x) = & \frac{(-1)^{k(k-1)/2} \, 2^k}{k!} \, \frac{1}{(2\pi i )^k} \, \oint \cdots \oint \frac{G(z_1, \cdots, z_k) \triangle (z_1^2, \cdots, z_k^2)^2}{\prod_{i=1}^k z_i^{2k-1}}\\
	& \times  q^{\frac{x}{2}\sum_{i=1}^k z_i} \, dz_1 \cdots z_k,
	\end{split}
	\end{equation}
	where $\Delta\left(z_1,\cdots,z_k\right)$ is defined as in (\ref{vandermonde}),
	
	\begin{equation}
	G(z_1, \cdots, z_k)= A_k(\tfrac{1}{2};z_1, \cdots, z_k) \prod_{i=1}^k \xx(\tfrac{1}{2}+z_i)^{-\frac{1}{2}} \prod_{1 \leq i \leq j \leq k} \z(1+z_i+z_j),
	\end{equation}
	and $A_k$ is the Euler product, absolutely convergent for $|\Re(z_i)|<\frac{1}{2},$ defined by 
	
	\begin{equation}\label{Q cont int}
	\begin{split}
	A_k(\tfrac{1}{2};z_1, \cdots, z_k)=& \prod_{\substack{P \text{ monic}\\ \text{irreducible}}} \prod_{1 \leq i \leq j \leq k} \Bigg(1- \frac{1}{|P|^{1+z_i+z_j}}\Bigg)\\
	& \times \frac{1}{2} \Bigg( \prod_{i=1}^k \Bigg(1-\frac{1}{|P|^{\frac{1}{2}+z_i}}\Bigg)^{-1}+ \prod_{i=1}^k \Bigg(1+\frac{1}{|P|^{\frac{1}{2}+z_i}}\Bigg)^{-1} \Bigg)
	\end{split}
	\end{equation}
	More generally, we have 
	
	\begin{equation}
	\begin{split}
	\sum_{P \in \p} & L(\tfrac{1}{2}+\alpha_1,\x_P) \cdots  L(\tfrac{1}{2}+\alpha_k,\x_P)\\
	& =\sum_{P\in \p} \prod_{i=1}^k \xx(\tfrac{1}{2}+\alpha_i)^{-\frac{1}{2}} |P|^{-\frac{1}{2}\sum_{i=1}^k\alpha_i}Q_k (\lo|P|,\alpha)(1+o(1))
	\end{split}
	\end{equation}
	where  
	
	\begin{equation}
	\begin{split}
	Q_k(x,\alpha) = & \frac{(-1)^{k(k-1)/2} \, 2^k}{k!} \, \frac{1}{(2\pi i )^k} \, \oint \cdots \oint \frac{G(z_1, \cdots, z_k) \triangle (z_1^2, \cdots, z_k^2)^2 \prod_{i=1}^k z_i}{ \prod_{i=1}^k \prod_{j=1}^k (z_j-\alpha_i) (z_j+\alpha_i)}\\
	& \times  q^{\frac{x}{2}\sum_{i=1}^k z_i} \, dz_1 \cdots z_k,
	\end{split}
	\end{equation}
	and the path of integration encloses the $\pm \alpha$'s.
\end{conjecture}

Note that, in the case when $k=1$ and $k=2$, this conjecture agrees with the results of Andrade and Keating \cite{a&kprimemean}. See Sections \ref{first moment} and \ref{second moment} for further details. 
\newline

The next conjecture is the translation for function fields of the ratios conjecture for quadratic Dirichlet $L$-functions associated with the character $\x_P$.

\begin{conjecture}\label{Our Ratios Conjecture}
	Suppose that the real part of $\alpha_k$ and $\gamma_k$are positive and that $q$ odd is the fixed cardinality of the finite field $\mathbb{F}_q$. Then with the same notation as before we have 
	
	\begin{equation}
	\begin{split}
	\sum_{P\in\p} & \frac{\prod_{k=1}^K L\left(\tfrac{1}{2}+\alpha_k,\x_P\right)}{\prod_{m=1}^Q L\left(\tfrac{1}{2}+\gamma_m,\x_P\right)}\\
	& = \sum_{P\in\p} \sum_{\varepsilon\in\{-1,1\}^k} \left|P\right|^{-\frac{1}{2} \sum_{k=1}^K \left(\varepsilon_k\alpha_k-\alpha_k\right)} \prod_{k=1}^K X\left(\tfrac{1}{2}+\tfrac{\alpha_k-\varepsilon_k\alpha_k}{2}\right) \\
	& \times Y\left(\varepsilon_1\alpha_1,\cdots,\varepsilon_K\alpha_K\right) A_\mathfrak{P}\left(\varepsilon_1\alpha_1,\cdots,\varepsilon_K\alpha_K\right)+ o\left(P\right),
	\end{split}
	\end{equation}
	where
	 
	\begin{equation}
	\begin{split}
	A_\mathfrak{P}&(\alpha;\gamma)=\prod_{\substack{P \text{ monic} \\ \text{irreducible}}} \frac{\prod_{j\le k\le K} \left(1-\frac{1}{|P|^{1+\alpha_j+\alpha_k}}\right) \prod_{m\le r\le Q} \left(1-\frac{1}{|P|^{1+\gamma_m+\gamma_r}}\right)}{\prod_{k=1}^K\prod_{m=1}^Q \left(1-\frac{1}{|P|^{1+\alpha_k+\gamma_m}}\right)}\\
	& \times \left(1+  \sum_{0< \sum_k a_k+\sum_m c_m \text{ is even}} \frac{\prod_{m=1}^Q \mu\left(P^{c_m}\right)}{|P|^{\sum_ka_k(\frac{1}{2}+\alpha_k)+\sum_m c_m(\frac{1}{2}+\gamma_m)}}\right)
	\end{split}
	\end{equation}
	and 
	
	\begin{equation}
	\begin{split}
	Y(\alpha;\gamma)= \frac{\prod_{j\le k\le K} \z\left(1+\alpha_j+\alpha_k\right) \prod_{m\le r\le Q} \z\left(1+\gamma_m+\gamma_r\right)}{\prod_{k=1}^K\prod_{m=1}^Q \z\left(1+\alpha_k+\gamma_m\right)}.
	\end{split}
	\end{equation}
\end{conjecture}

In the following sections we present the details of how to arrive at these conjectures.

\section{Integral moments of $L$-functions over prime polynomials}\label{Integral moments}

In this section, we present the details of the recipe for conjecturing moments of the family of quadratic Dirichlet $L$-function $L(s,\chi_{P})$ associated to hyperelliptic curves of genus $g$ over fixed finite field $\mathbb{F}_q$ as $g\to\infty$. As in Andrade and Keating \cite{a&kConInMo}, we will adjust the recipe first presented in \cite{CFKRS} to the function field setting.  
\newline

Let $P\in \p.$ For a fixed $k$, we aim to obtain an asymptotic expression for 

\begin{equation}\label{L-k}
\sum_{P\in\p} L\left(\tfrac{1}{2},\x_P\right)^k,
\end{equation}
as $g \to \infty.$ In order to achieve this we consider the more general expression obtained by introducing small shifts, say $\alpha_1, \cdots, \alpha_k$

\begin{equation}
\sum_{P\in\p} L\left(\tfrac{1}{2}+\alpha_1,\x_P\right) \cdots L\left(\tfrac{1}{2}+\alpha_k,\x_P\right).
\end{equation}
\newline

By introducing the shifts it helps to reveal the hidden structures in the form of symmetries. Moreover, the calculations are simplified by the removal of higher order poles. In the end, letting each $\alpha_1,\cdots,\alpha_k$ tend to $0$ will provide an asymptotic formula for (\ref{L-k}). 
\newline

\subsection{Analogies between classical $L$-functions and $L$-functions over function fields}$\text{\color{white}kgm}$\\

The first step to obtain a conjecture for the integral moments of $L$-functions of any family is the use of  the approximate functional equation. Thus, the ``approximate" functional equation for the $L$-function attached to the character $\x_P$ is given by 

\begin{equation}\label{xapro}
L(s,\x_P)= \sum_{\substack{n \text{ monic}\\ \text{deg}(n)\leq g}} \frac{\x_P(n)}{|n|^s} + \xx_p(s) \sum_{\substack{n \text{ monic}\\ \text{deg}(n)\leq g-1}} \frac{\x_P(n)}{|n|^{1-s}},
\end{equation}
where $P\in\p$ and $\xx_P(s)=q^{g(1-2s)}.$ Note that $\xx_P(s)$ can also be re-written as, 

\begin{equation}\label{xx_P}
\xx_P(s)=|P|^{\frac{1}{2}-s} \xx(s),
\end{equation}
where $\xx(s)=q^{-\frac{1}{2}+s}$ corresponds to the gamma factor that appears in the classical quadratic $L$-functions.
\newline

The following lemma, which is easy to check, will be used in the recipe. The following lemma makes the analogy between the function field case and the number field case more apparent.

\begin{lemma}\label{x_P=x}
	We have that, 
	
	\begin{equation}
	\xx_P(s)^{\frac{1}{2}}=\xx_P(1-s)^{-\frac{1}{2}},
	\end{equation}  
	and
	
	\begin{equation}
	\xx_P(s) \, \xx_P(1-s)=1.
	\end{equation}  
\end{lemma}

%

Consider the following completed $L$-function

\begin{equation}\label{Z_L}
Z_{\lL}(s,\x_P)=\xx_P(s)^{-\frac{1}{2}}L(s,\x_P).
\end{equation}
We will apply the recipe to this completed $L$-function, since it simplifies the calculations, and satisfies a more symmetric functional equation given by the next lemma.

\begin{lemma}
	Let $Z_{\lL}(s,\x_P)$ be the $Z$-function defined above, then we have the following functional equation,
	 
	\begin{equation}
	Z_{\lL}(s,\x_P)=Z_{\lL}(1-s,\x_P).
	\end{equation}
\end{lemma}
\begin{proof}
Direct from the definition of $Z_{\lL}(s,\x_P)$ and Lemma \ref{x_P=x}.
\end{proof}

Now, let 

\begin{equation}
L_P(s)= \sum_{P\in\p} Z(s;\alpha_1, \cdots, \alpha_k),
\end{equation}
be the $k$-shifted moment, with

\begin{equation}
Z(s;\alpha_1, \cdots, \alpha_k)= \prod_{i=1}^k\, Z_{\lL}(s+\alpha_i,\x_P).
\end{equation}
Using the ``approximate" functional equation (\ref{xapro}) and Lemma \ref{x_P=x} we have
 
\begin{equation} \label{prime_xapro}
Z_{\lL}(s,\x_P)= \xx_P(s)^{-\frac{1}{2}} \sum_{\substack{n \text{ monic}\\ \text{deg}(n)\leq g}} \frac{\x_P(n)}{|n|^s} \, + \,  \xx_p(1-s)^{-\frac{1}{2}} \sum_{\substack{n \text{ monic}\\ \text{deg}(n)\leq g-1}} \frac{\x_P(n)}{|n|^{1-s}}.
\end{equation}

\subsection{Adapting the CFKRS recipe for the function field case}$\text{\color{white}kgm}$\\

We present the steps of the recipe which follows from \cite{CFKRS} and \cite{a&kConInMo} with the necessary modifications for the family of $L(s,\chi_{P})$.
\newline

\begin{enumerate}
	\item Write the product of $k$-shifted $L$-functions.

\begin{equation}
Z(\tfrac{1}{2};\alpha_1,\cdots, \alpha_k)= Z_\lL(\tfrac{1}{2}+\alpha_1,\x_P) \cdots Z_\lL(\tfrac{1}{2}+\alpha_k,\x_P).
\end{equation}\\

\item Replace each $L$-function with the two terms from its approximate functional equation (\ref{xapro}) with $s=1/2+\alpha_i$.

%

\begin{equation}\label{x_p(n)}
\begin{split}
Z(\tfrac{1}{2}; \alpha_1, \cdots, \alpha_k) &= \sum_{\varepsilon_i=\pm 1} \prod_{i=1}^k \xx_P(\tfrac{1}{2}+\varepsilon_i\alpha_i)^{-\frac{1}{2}} \sum_{\substack{n_1, \cdots, n_k \\ \text{deg}(n_i)\leq f(\epsilon_{i})}} \frac{\x_P(n_1 \cdots n_k)}{\prod_{i=1}^{k}|n_i|^{\frac{1}{2}+\varepsilon_i\alpha_i}}.
\end{split}
\end{equation}
where $f(1)=g,$ and $f(-1)=g-1.$

\item Replace each product of $\e_f$-factors by its expected value when averaged over $\p$.\\

In our case $\varepsilon_f$-factors are equal to $1.$ Thus the product will not appear and will not affect the result. \\ 

\item Replace each summand by its expected value when averaged over $\p$. \\

We need first to average over all primes $P\in\p.$ The next lemma gives the orthogonality relation for these quadratic Dirichlet characters over function fields.

\begin{lemma}\label{expected_value}
	\begin{equation}
	\lim_{\text{deg}(P) \to \infty} \frac{1}{\# \p} \sum_{P\in\p}\x_P(n)= 
	\begin{cases} 
	1 &\mbox{if } n=\Box \\ 
	0 & \mbox{otherwise. } \end{cases}
	\end{equation}
\end{lemma}
\begin{proof}
	Consider the case when $n=\Box,$ then we have 
	
	\begin{equation}
	\sum_{P\in\p} \x_P(n)=\sum_{P\in\p}\x_P(l^2) = \sum_{\substack{P\in\p \\ P\nmid l}} 1, 
	\end{equation}
	since we are summing over primes of degree $2g+1$ and $P\nmid l,$ and $\text{deg}(l)\le 2g,$ which means that we are counting all primes of degree $2g+1,$ thus
	
	\begin{equation}
	\sum_{\substack{P\in \p \\ P\nmid l}} 1 = \# \p.
	\end{equation}
	Hence if $n$ is a square of a polynomial,
	
	\begin{equation}
	\lim_{\text{deg}(P)\to\infty} \frac{1}{\#\p} \sum_{P\in\p} \x_P(n) = 1.
	\end{equation}
	\newline

	It remains to consider the case when $n\neq\Box,$ Rudnick \cite{Rudnick_trace} shows that
	
	\begin{equation}
	\left|\sum_{P\in\p}\x_P(n)\right| \ll \frac{|P|^{\frac{1}{2}}}{\lo|P|} \text{deg}(n),
	\end{equation}
	and from Polynomial Prime Theorem (\ref{PNT}) we have
	
	
	\begin{equation}
	\begin{split}
	\frac{1}{\#\p} \sum_{P\in\p} \x_P(n) 
	& \ll |P|^{-\frac{1}{2}} \text{deg}(n).
	\end{split}
	\end{equation}
	Hence if $n$ is not a square of a polynomial we have that 
	
	\begin{equation}
	\lim_{\text{deg}(P)\to\infty} \frac{1}{\#\p} \sum_{P\in\p} \x_P(n) = 0.
	\end{equation}
\end{proof}

Using Lemma \ref{expected_value} we can average the summand in (\ref{x_p(n)}), that is

\begin{equation}
\begin{split}
\lim_{\text{deg}(P) \to \infty} \frac{1}{\# \p} \sum_{P\in\p}& \sum_{\substack{n_1, \cdots, n_k }} \frac{\x_P(n_1 \cdots n_k)}{\prod_{i=1}^{k}|n_i|^{\frac{1}{2}+\varepsilon_i\alpha_i}}\\
& = \sum_{m \text{ monic}} \sum_{\substack{n_1, \cdots, n_k  \\ n_1\cdots n_k=m^2}} \frac{1}{\prod_{i=1}^{k}|n_i|^{\frac{1}{2}+\varepsilon_i\alpha_i}}.
\end{split}
\end{equation}

\item Let each $n_1, \cdots, n_k$ to be monic polynomials, and call the total result $M_f(s, \alpha_1,\cdots, \alpha_k)$ to produce the desired conjecture.\\

If we let

\begin{equation}
R_k\left(\tfrac{1}{2};\varepsilon_1\alpha_1,\cdots, \varepsilon_k\alpha_k\right)= \sum_{m \text{ monic}} \sum_{\substack{n_1, \cdots, n_k \\ n_i \text{ monic} \\ n_1\cdots n_k=m^2}} \frac{1}{\prod_{i=1}^{k}|n_i|^{\frac{1}{2}+\varepsilon_i\alpha_i}},
\end{equation} 
then the extended sum produced by the recipe is 

\begin{equation}\label{M(1/2)}
M\left(\tfrac{1}{2};\alpha_1, \cdots, \alpha_k\right) = \sum_{\varepsilon_i=\pm 1} \prod_{i=1}^k \x_P\left(\tfrac{1}{2}+\varepsilon_i \alpha_i\right)^{-\frac{1}{2}} R_k\left(\tfrac{1}{2};\varepsilon_1\alpha_1,\cdots, \varepsilon_k\alpha_k\right).
\end{equation}
\\

\item The conclusion is \\
\begin{equation}\label{conjecture_form1}
\begin{split}
\sum_{P\in \p} Z\left(\tfrac{1}{2};\alpha_1,\cdots, \alpha_k\right) & =\sum_{P\in\p} M\left(\tfrac{1}{2},\alpha_1,\cdots, \alpha_k\right) \, \left( 1 + o(1)\right).\\
\end{split}
\end{equation}

\end{enumerate}
	
\subsection{Putting the conjecture in a more useful form}$\text{\color{white}kgm}$\\
 
  In this section we put the conjecture (\ref{conjecture_form1}) in a more useful form, we write $R_k$ as an Euler product, then factors out the appropriate $\z(s)$-factors. Let 
 
 \begin{equation}
 \psi(x):= \sum_{\substack{n_1, \cdots, n_k \\ n_i \text{ monic} \\n_1 \cdots n_k=x }} \frac{1}{|n_1|^{s+\alpha_1}\cdots |n_k|^{s+\alpha_k}},
 \end{equation} 
 then it is easy to see that $\psi(m^2)$ is multiplicative on $m$. We can write $R_k(s;\alpha_1,\cdots,\alpha_k)$ as
 
 \begin{equation}
 \begin{split}
 R_k(s;\alpha_1,\cdots,\alpha_k) 
 &= \sum_{m \text{ monic}} \psi(m^2)\\
 & = \prod_{\substack{P \text{ monic} \\ \text{irreducible}}} \Bigg(1+ \sum_{j=1}^\infty \psi(P^{2j})\bigg),
 \end{split}
 \end{equation}
 where 
 \begin{equation}\label{psi2j}
 \psi(P^{2j}) = \sum_{\substack{n_1, \cdots, n_k \\ n_i \text{ monic} \\n_1 \cdots n_k=P^{2j}}} \frac{1}{|n_1|^{s+\alpha_1}\cdots |n_k|^{s+\alpha_k}}.
 \end{equation}
 \newline

 Since we have $n_1 \cdots n_k=P^{2j}$,  then for each $i=1, \cdots, k,$ write $n_i$ as $n_i=P^{e_i},$ for some $e_i\ge 0$ and $e_1+\cdots+e_k=2j,$ and (\ref{psi2j}) becomes
 
 \begin{equation}
 \begin{split}
 \psi(P^{2j}) 
 &= \sum_{\substack{e_1,\cdots,e_k\ge 0 \\ e_1+\cdots+e_k=2j}} \prod_{i=1}^k \frac{1}{|P|^{e_i(s+\alpha_i)}},
 \end{split}
 \end{equation}
 and so, we have 
 
 \begin{equation}
 \begin{split}
 R_k(s;\alpha_1,\cdots,\alpha_k) 
 & =  \prod_{\substack{P \text{ monic} \\ \text{irreducible}}}\Bigg(1+\sum_{j=1}^\infty\sum_{\substack{e_1,\cdots,e_k\ge 0 \\ e_1+\cdots+e_k=2j}} \prod_{i=1}^k \frac{1}{|P|^{e_i(s+\alpha_i)}}\Bigg).
 \end{split}
 \end{equation}
 One can see that when $\alpha_i=0$ and $s=1/2$, the poles only arise from the terms with $e_1+\cdots+e_k=2$. Define $R_{k,P}(s;\alpha_1,\cdots,\alpha_k)$ to be as follow
 
 \begin{equation}\label{R_kP}
 \begin{split}
 R_{k,P}(s;\alpha_1,\cdots,\alpha_k)&= 1+\sum_{j=1}^\infty\sum_{\substack{e_1,\cdots,e_k\ge 0 \\ e_1+\cdots+e_k=2j}} \prod_{i=1}^k \frac{1}{|P|^{e_i(s+\alpha_i)}}\\
 &= 1+\sum_{\substack{e_1,\cdots,e_k\ge 0 \\ e_1+\cdots+e_k=2}} \prod_{i=1}^k \frac{1}{|P|^{e_i(s+\alpha_i)}}+ \text{ (lower order terms)}\\
 & = 1 + \sum_{1\le i\le j\le k}  \frac{1}{|P|^{2s+\alpha_i+\alpha_j}} + O\Big( |P|^{-4s+\epsilon}\Big),
 \end{split}
 \end{equation}
 for $\mathfrak{R}(\alpha_i)$ small enough (see \cite{CFKRS} for more details).
%
And so, we have  
 
 \begin{equation}\label{R_kp product}
 \begin{split}
 R_{k,P} & (s;\alpha_1,\cdots,\alpha_k)\\
 & = \prod_{1\le i\le j\le k} \left(1+\frac{1}{|P|^{2s+\alpha_i+\alpha_j}}\right) \times \left(1+O\left(|P|^{-4s+\epsilon}\right)\right).
 \end{split}
 \end{equation}
 Recall that,

 \begin{equation}\label{z2s/z4s}
 \begin{split}
 \frac{\z(2s)}{\z(4s)} 
 & =  \prod_{\substack{P \text{ monic} \\ \text{irreducible}}} \Big(1+\frac{1}{|P|^{2s}}\Big)\\
 \end{split}
 \end{equation}
 has a simple pole as $s=1/2.$ Therefore
  
 \begin{equation}
 \prod_{\substack{P \text{ monic} \\ \text{irreducible}}}\Big(1+O\big(|P|^{-4s+\epsilon}\big)\Big)
 \end{equation}
 is analytic in $\mathfrak{R}(s)>1/4,$ and $\prod_{P} R_{k,P}$ has a pole at $s=1/2$ of order $k(k+1)/2$ if $\alpha_i=0$ for all $i=1,\cdots,k.$ It remains to factor out the appropriate zeta-factors. Since we have
 
 \begin{equation}
 R_k(s;\alpha_1,\cdots,\alpha_k) =\prod_{\substack{P \text{ monic} \\ \text{irreducible}}} R_{k,P}(s;\alpha_1,\cdots,\alpha_k),
 \end{equation}
 then from (\ref{R_kp product}) and (\ref{z2s/z4s}) we can write
 
 \begin{equation}\label{R_k(s)}
 \begin{split}
 R_k(s;\alpha_1,\cdots&,\alpha_k)
  = \prod_{1\leq i\leq j \leq k} \z(2s+\alpha_i+\alpha_j) A_k(s;\alpha_1,\cdots,\alpha_k),
 \end{split}
 \end{equation}
 where
 
 \begin{equation}
 \begin{split}
 A_k(s;\alpha_1&,\cdots,\alpha_k) \\
 &= \prod_{\substack{P \text{ monic} \\ \text{irreducible}}} \left( R_{k,P}(s;\alpha_1,\cdots,\alpha_k) \prod_{1\leq i\leq j \leq k} \left(1-\frac{1}{|P|^{2s+\alpha_i+\alpha_j}}\right) \right).
 \end{split}
 \end{equation}
 \newline

 Notice that for some $\delta>0$ and for all $\alpha_i$'s in some sufficiently small neighbourhood of $0,$ $A_k$ is an absolutely convergent Dirichlet series for $\mathfrak{R}(s)>1/2+\delta.$ Combining (\ref{M(1/2)}) and (\ref{R_k(s)}), we have 
 
 \begin{equation}
 \begin{split}
  M\left(\tfrac{1}{2};\alpha_1, \cdots, \alpha_k\right)  = & \sum_{\varepsilon_i=\pm 1} \prod_{i=1}^k \xx_P\left(\tfrac{1}{2}+\varepsilon_i \alpha_i\right)^{-\frac{1}{2}} \prod_{1\leq i\leq j \leq k} \z(1+\alpha_i+\alpha_j) \\
 & \times A_k\left(\tfrac{1}{2};\varepsilon_1\alpha_1,\cdots,\varepsilon_k\alpha_k\right).
 \end{split}
 \end{equation}
 \newline

 Hence, 
 
 \begin{equation}
 \begin{split}
 \sum_{P\in \p} & Z\left(\tfrac{1}{2};\alpha_1, \cdots, \alpha_k\right) \\
  = & \sum_{P\in \p} \sum_{\varepsilon_i=\pm 1} \prod_{i=1}^k \xx_P\left(\tfrac{1}{2}+\varepsilon_i \alpha_i\right)^{-\frac{1}{2}} A_k\left(\tfrac{1}{2};\varepsilon_1\alpha_1,\cdots,\varepsilon_k\alpha_k\right)\\
 & \times \prod_{1\leq i\leq j \leq k} \z(1+\alpha_i+\alpha_j) \left(1+o\left(1\right)\right).
 \end{split}
 \end{equation}
 \newline

 From the definition of $\xx_P(s)$ in (\ref{xx_P}), we have 
 
 \begin{equation}
 \begin{split}
 \xx_P\left(\tfrac{1}{2}+\varepsilon_i \alpha_i\right)^{-\frac{1}{2}}
 & = |P|^{\frac{\varepsilon_i\alpha_i}{2}}\xx\left(\tfrac{1}{2}+\varepsilon_i \alpha_i\right)^{-\frac{1}{2}}.\\
 \end{split}
 \end{equation}
 \newline

 Hence,
 \begin{equation}\label{Z=R}
 \begin{split}
 \sum_{P\in \p} Z\big(\tfrac{1}{2}&;\alpha_1, \cdots, \alpha_k\big) \\
  = & \sum_{\varepsilon_i=\pm 1} \prod_{i=1}^k \xx\left(\tfrac{1}{2}+\varepsilon_i \alpha_i\right)^{-\frac{1}{2}} \sum_{P\in \p} R_k\left(\tfrac{1}{2};\varepsilon_1\alpha_1,\cdots,\varepsilon_k\alpha_k\right)\\
 & \times |P|^{\frac{1}{2}\sum_{i=1}^k \varepsilon_i\alpha_i}\left(1+o\big(1\big)\right).
 \end{split}
 \end{equation}
 \newline

 We finish this section writing $A_k$ as an Euler product in the following lemma.
 
 \begin{lemma}\label{A_k} We have 
 	
 	\begin{equation}
 	\begin{split}
 	A_k\left(\tfrac{1}{2};\alpha_1,\cdots,\alpha_k\right) & =  \prod_{\substack{P \text{ monic} \\ \text{irreducible}}} \prod_{1\le i\le j\le k}\left(1-\frac{1}{|P|^{1+\alpha_i+\alpha_j}}\right)\\
 	& \times \frac{1}{2} \left(\prod_{i=1}^k \left(1-\frac{1}{|P|^{1/2+\alpha_i}}\right)^{-1}+ \prod_{i=1}^k \left(1+\frac{1}{|P|^{1/2+\alpha_i}}\right)^{-1}\right). \\
 	&\\
 	\end{split}
 	\end{equation}
 \end{lemma}
 \begin{proof}
 	We define 
 	
 	\begin{equation}\label{A_k=R_kp}
 	\begin{split}
 	A_k(\tfrac{1}{2};\alpha_1,&\cdots,\alpha_k)\\
 	&  = \prod_{\substack{P \text{ monic} \\ \text{irreducible}}} R_{k,P}(\tfrac{1}{2};\alpha_1,\cdots,\alpha_k) \prod_{1\leq i\leq j \leq k} \Big(1-\frac{1}{|P|^{1+\alpha_i+\alpha_j}}\Big), 
 	\end{split}
 	\end{equation}
 	then from (\ref{R_kP}) we can write,
 	\begin{equation}
 	\begin{split}
 	& A_k(\tfrac{1}{2};\alpha_1,\cdots,\alpha_k)\\
 	&  = \prod_{\substack{P \text{ monic} \\ \text{irreducible}}} \prod_{1\leq i\leq j \leq k} \Big(1-\frac{1}{|P|^{1+\alpha_i+\alpha_j}}\Big) \Bigg(1+\sum_{j=1}^\infty\sum_{\substack{e_1,\cdots,e_k\ge 0 \\ e_1+\cdots+e_k=2j}} \prod_{i=1}^k \frac{1}{|P|^{e_i(1/2+\alpha_i)}}\Bigg) , 
 	\end{split}
 	\end{equation}
 	by simplifying the second brackets we obtain the result in the lemma, that is,  
 	\begin{equation}
 	\begin{split}
 	1+&\sum_{j=1}^\infty \sum_{\substack{e_1,\cdots,e_k\ge 0 \\ e_1+\cdots+e_k=2j}} \prod_{i=1}^k \Bigg(\frac{1}{|P|^{(1/2+\alpha_i)}}\Bigg)^{e_i} \\
 	& = \sum_{j=0}^\infty\frac{1}{2} \Bigg(2 \sum_{\substack{e_1,\cdots,e_k\ge 0 \\ e_1+\cdots+e_k=2j}} \prod_{i=1}^k \Bigg(\frac{1}{|P|^{(1/2+\alpha_i)}}\Bigg)^{e_i}\Bigg)\\
 	& = \frac{1}{2} \Bigg( \prod_{i=1}^k \sum_{e_1= 0 }^\infty \Bigg(\frac{1}{|P|^{(1/2+\alpha_i)}}\Bigg)^{e_i} + \prod_{i=1}^k \sum_{e_1= 0 }^\infty (-1)^{e_1+\cdots+e_k} \Bigg(\frac{1}{|P|^{(1/2+\alpha_i)}}\Bigg)^{e_i}\Bigg)\\
 	& = \frac{1}{2} \Bigg(\prod_{i=1}^k \Big(1-\frac{1}{|P|^{1/2+\alpha_i}}\Big)^{-1}+ \prod_{i=1}^k \Big(1+\frac{1}{|P|^{1/2+\alpha_i}}\Big)^{-1}\Bigg).
 	\end{split}
 	\end{equation}
 \end{proof}

\subsection{The contour integral representation of the conjecture}$\text{\color{white}kgm}$\\

We begin this section with Lemma 2.5.2 from \cite{CFKRS}, which helps to write our conjecture as a contour integral. 

\begin{lemma}\label{CFKRS}
	Suppose $F$ is a symmetric function of $k$ variables, regular near $(0,\cdots,0)$, and that $f(s)$ has a simple pole $s=0$ of residue $1$ and is otherwise analytic in a neighbourhood of $s=0,$ and let 
	
	\begin{equation}
	K(a_1,\cdots,a_k)=F(a_1,\cdots,a_k) \prod_{1\le i\le j\le k} f(a_i+a_j),
	\end{equation}
	or
	
	\begin{equation}
	K(a_1,\cdots,a_k)=F(a_1,\cdots,a_k) \prod_{1\le i< j\le k} f(a_i+a_j).
	\end{equation}
	If $\alpha_i+\alpha_j$ are contained in the region of analyticity of $f(s),$ then
	
	\begin{equation}
	\begin{split}
	\sum_{\varepsilon_i=\pm 1} K(\varepsilon_1a_1,\cdots,\varepsilon_k a_k) &= \frac{(-1)^{k(k-1)/2}}{(2\pi i)^k} \frac{2^k}{k!}\oint \cdots \oint K(z_1,\cdots,z_k)\\& \times \frac{\Delta (z_1^2,\cdots,z_k^2)^2 \prod_{i=1}^kz_i}{\prod_{i=1}^k \prod_{j=1}^k (z_i-\alpha_j)(z_i+\alpha_j)}dz_1\cdots dz_k,
	\end{split}
	\end{equation}
	and
	
	\begin{equation}
	\begin{split}
	\sum_{\varepsilon_i=\pm 1} \Bigg(\prod_{i=1}^k \varepsilon_i\Bigg) K(\varepsilon_1a_1,\cdots,& \varepsilon_k a_k)\\
	& = \frac{(-1)^{k(k-1)/2}}{(2\pi i)^k} \frac{2^k}{k!}\oint \cdots \oint K(z_1,\cdots,z_k)\\& \times \frac{\Delta (z_1^2,\cdots,z_k^2)^2 \prod_{i=1}^k \alpha_i}{\prod_{i=1}^k \prod_{j=1}^k (z_i-\alpha_j)(z_i+\alpha_j)}dz_1\cdots dz_k,
	\end{split}
	\end{equation}
	where the path of the integration encloses the $\pm\alpha_i$'s.
\end{lemma}

Recall that,

\begin{equation}\label{Z=L}
\begin{split}
\sum_{P\in \p}& Z\left(\tfrac{1}{2};\alpha_1,\cdots,\alpha_k\right)\\
& =\sum_{P\in \p}\prod_{i=1}^k \xx_P\left(\tfrac{1}{2}+\alpha_i\right)^{-\frac{1}{2}} L\left(\tfrac{1}{2}+\alpha_i,\x_P\right), 
\end{split}
\end{equation}
%
where $\xx_P(s)$ is defined in (\ref{xx_P}). Since $\xx_P(\tfrac{1}{2}+\alpha_i)^{-\frac{1}{2}}$ does not depend on $P$, we can factor out it, and from (\ref{Z=L}) and (\ref{Z=R}) we have 

%
\begin{equation}
\begin{split}
 \sum_{P\in \p} &\prod_{i=1}^k L\left(\tfrac{1}{2}+\alpha_i,\x_P\right)\\
= & \sum_{P\in \p} |P|^{-\frac{1}{2}\sum_{i=0}^k \alpha_i}\prod_{i=1}^k \xx\left(\tfrac{1}{2}+\alpha_i\right)^{\frac{1}{2}} \sum_{\varepsilon_i=\pm 1}  \prod_{i=1}^k \xx\left(\tfrac{1}{2}+\varepsilon_i\alpha_i\right)^{-\frac{1}{2}}\\
& \text{\color{white}gkm} \times  A_k\left(\tfrac{1}{2};\alpha_1,\cdots,\alpha_k\right) |P|^{\frac{1}{2}\sum_{i=0}^k \varepsilon_i \alpha_i} \\
& \text{\color{white}gkmfl,nfk} \times  \prod_{1\le i< j\le k}\z(1+\varepsilon_i\alpha_i+\varepsilon_j\alpha_j)\Big(1+o\big(1\big)\Big).
\end{split}
\end{equation}
\newline

From each term in the second product we factor out $(\log q)^{-1}$ to get

\begin{equation}\label{A_klog}
\begin{split}
\sum_{P\in \p} & \prod_{i=1}^k L\left(\tfrac{1}{2}+\alpha_i,\x_P\right)\\
& = \sum_{P\in \p} \frac{|P|^{-\frac{1}{2}\sum_{i=0}^k \alpha_i}\prod_{i=1}^k \xx\left(\tfrac{1}{2}+\alpha_i\right)^{\frac{1}{2}}}{(\log q)^{k(k+1)/2}} \sum_{\varepsilon_i=\pm 1}  \prod_{i=1}^k \xx\left(\tfrac{1}{2}+\varepsilon_i\alpha_i\right)^{-\frac{1}{2}}\\
&  \text{\color{white} gd} \times A_k\left(\tfrac{1}{2};\alpha_1,\cdots,\alpha_k\right) |P|^{\frac{1}{2}\sum_{i=0}^k \varepsilon_i \alpha_i} \\
& \text{\color{white} gd} \times \prod_{1\le i< j\le k}\z(1+\varepsilon_i\alpha_i+\varepsilon_j\alpha_j)(\log q) \Big(1+o\big(1\big)\Big).
\end{split}
\end{equation}
\newline

Now, call
\begin{equation}
F(\alpha_1,\cdots,\alpha_k)= \prod_{i=1}^k\xx(\tfrac{1}{2}+\alpha_i)^{-\frac{1}{2}} A_k(\tfrac{1}{2};\alpha_i,\cdots,\alpha_k) |P|^{\frac{1}{2}\sum_{i=1}^k\alpha_i},
\end{equation}
and

\begin{equation}
f(s)=\z(1+s)\log q \; \text{ \;   and so   \;}\; f(\alpha_i+\alpha_j)=\z(1+\alpha_i+\alpha_j)\log q,
\end{equation}
where $f(s)$ has a simple pole at $s=0$ with residue $1$. 
\newline

If we denote 

\begin{equation}
K(\alpha_1,\cdots,\alpha_k)= F(\alpha_1,\cdots,\alpha_k) \prod_{1\le i\le j\le k}f(\alpha_i+\alpha_j),
\end{equation}
then (\ref{A_klog}) can be written as 

\begin{equation}
\begin{split}
\sum_{P\in \p} \prod_{i=1}^k  L&\left(\tfrac{1}{2}+\alpha_i,\x_P\right)\\
%
&=\sum_{P\in \p} \frac{\prod_{i=1}^k|P|^{-\frac{1}{2}\sum_{i=0}^k \alpha_i} \xx\left(\tfrac{1}{2}+\alpha_i \right)^{\frac{1}{2}}}{(\log q)^{k(k+1)/2}}  \\
& \text{\color{white}fklj rgkohrk fljh}\times \sum_{\varepsilon_i=\pm 1} K(\varepsilon_1\alpha_1,\cdots,\varepsilon_k\alpha_k) \left(1+o\left(1\right)\right).\\
\end{split}
\end{equation}
\newline

Using Lemma \ref{CFKRS} we have

\begin{equation}\label{6.4.9}
\begin{split}
&\sum_{P\in \p} \prod_{i=1}^k L\left(\tfrac{1}{2}+\alpha_i,\x_P\right)\\
&=\sum_{P\in \p}  \frac{\prod_{i=1}^k|P|^{-\frac{1}{2}\sum_{i=0}^k \alpha_i} \xx\left(\tfrac{1}{2}+\alpha_i\right)^{\frac{1}{2}}}{(\log q)^{k(k+1)/2}}  \frac{(-1)^{k(k-1)/2}}{(2\pi i)^k} \frac{2^k}{k!}\color{white}\text{lkfjgnfnf gdmnbjnndlkn}\\
& \text{\color{white}gfjmj}\times \oint \cdots \oint K(z_1,\cdots,z_k) \frac{\Delta (z_1^2,\cdots,z_k^2)^2 \prod_{i=1}^kz_i}{\prod_{i=1}^k \prod_{j=1}^k (z_i-\alpha_j)(z_i+\alpha_j)}dz_1\cdots dz_k\\
& \text{\color{white}gfjmj} \times \Big(1+o\big(1\big)\Big)\\
&\\
& =\sum_{P\in \p} \prod_{i=1}^k |P|^{-\frac{1}{2}\sum_{i=0}^k \alpha_i} \xx\left(\tfrac{1}{2}+\alpha_i\right)^{\frac{1}{2}} \frac{(-1)^{k(k-1)/2}}{(2\pi i)^k} \frac{2^k}{k!}\\
& \text{\color{white}fk lnkdr jmh} \times   \oint \cdots \oint F(z_1,\cdots,z_k)  \prod_{1\le i \le j\le k}\z(1+\varepsilon_i\alpha_i+\varepsilon_j\alpha_j)\\
& \text{\color{white}fk f;lnkdr jmhgj} \times \frac{\Delta (z_1^2,\cdots,z_k^2)^2 \prod_{i=1}^kz_i}{\prod_{i=1}^k \prod_{j=1}^k (z_i-\alpha_j)(z_i+\alpha_j)}dz_1\cdots dz_k +o\big(|P|\big)\\
&\\
& =\sum_{P\in \p} \prod_{i=1}^k |P|^{-\frac{1}{2}\sum_{i=0}^k \alpha_i} \xx\left(\tfrac{1}{2}+\alpha_i\right)^{\frac{1}{2}} \frac{(-1)^{k(k-1)/2}}{(2\pi i)^k} \frac{2^k}{k!}\\
& \text{\color{white}fk lnh} \times   \oint \cdots \oint K(z_1,\cdots,z_k) \frac{\Delta (z_1^2,\cdots,z_k^2)^2 \prod_{i=1}^kz_i}{\prod_{i=1}^k \prod_{j=1}^k (z_i-\alpha_j)(z_i+\alpha_j)}dz_1\cdots dz_k \\
& \text{\color{white}fk lnh}  +o\big(|P|\big),
\end{split}
\end{equation}
with 
\begin{equation}
K(z_1,\cdots,z_k) = F(z_1,\cdots,z_k)  \prod_{1\le i \le j\le k}\z(1+\varepsilon_i\alpha_i+\varepsilon_j\alpha_j).
\end{equation}
\newline

Moreover, if we denote 

\begin{equation}
\begin{split}
G(z_1,\cdots,z_k)&= \prod_{i=1}^k\xx\left(\tfrac{1}{2}+\alpha_i\right)^{-\frac{1}{2}} A_k\left(\tfrac{1}{2};\alpha_i,\cdots,\alpha_k\right) \prod_{1\le i \le j\le k}\z(1+z_i+z_j)
\end{split}
\end{equation}
then (\ref{6.4.9}) becomes

\begin{equation}
\begin{split}
&\sum_{P\in \p} \prod_{i=1}^k |P|^{-\frac{1}{2}\sum_{i=0}^k \alpha_i} \xx\left(\tfrac{1}{2}+\alpha_i\right)^{\frac{1}{2}} \frac{(-1)^{k(k-1)/2}}{(2\pi i)^k} \frac{2^k}{k!} \\
& \times \oint \cdots \oint G(z_1,\cdots,z_k)  |P|^{\frac{1}{2}\sum_{i=0}^k z_i}\frac{\Delta (z_1^2,\cdots,z_k^2)^2 \prod_{i=1}^kz_i}{\prod_{i=1}^k \prod_{j=1}^k (z_i-\alpha_j)(z_i+\alpha_j)}dz_1\cdots dz_k\\
& +o\big(|P|\big).
\end{split}
\end{equation}

Now, letting $\alpha_i\to 0,$ we have 

\begin{equation}
\begin{split}
&\sum_{P\in \p} L(\tfrac{1}{2},\x_P)^k\\
& \text{\color{white} dgj}=\sum_{P\in \p}  \frac{(-1)^{k(k-1)/2}}{(2\pi i)^k} \frac{2^k}{k!}  \oint \cdots \oint G(z_1,\cdots,z_k)  |P|^{\frac{1}{2}\sum_{i=0}^k z_i} \\
& \text{\color{white} dgdgkngfkhoi lktjmyh}\times \frac{\Delta (z_1^2,\cdots,z_k^2)^2 \prod_{i=1}^kz_i}{\prod_{i=1}^k z_i^{2k}}dz_1\cdots dz_k +o\big(|P|\big).
\end{split}
\end{equation}
Calling  

\begin{equation}
\begin{split}
Q_k(x) = \frac{(-1)^{k(k-1)/2}}{(2\pi i)^k} & \frac{2^k}{k!}  \oint \cdots \oint G(z_1,\cdots,z_k) \\
\times &q^{\frac{x}{2}\sum_{i=0}^k z_i}\frac{\Delta (z_1^2,\cdots,z_k^2)^2 \prod_{i=1}^kz_i}{\prod_{i=1}^k z_i^{2k}}dz_1\cdots dz_k,
\end{split}
\end{equation}
we obtain the formula of the Conjecture \ref{ourconjecture}, i.e.,
\begin{equation}
\sum_{P\in \p} L(\tfrac{1}{2},\x_P)^k= \sum_{P\in \p} Q_k(\lo|P|)\left(1+o\left(1\right)\right).
\end{equation}

\section{Some conjectural formulae for moments of $L$-functions associated with $\x_P$}\label{conjecture formula}

We use Conjecture \ref{ourconjecture} to obtain explicit conjectural values for several moments of quadratic Dirichlet $L$-functions associated to $\x_P$ over function fields.

\subsection{First moment}\label{first moment} $\text{\color{white}kgm}$\\

We will use Conjecture \ref{ourconjecture} when $k=1$ to compute the first moment of our family of $L$-functions, then compare the result with that of Andrade and Keating proved in \cite{a&kprimemean}. For $k=1$ the formula in Conjecture \ref{ourconjecture} gives 

\begin{equation}
\sum_{P\in \p} L\left(\tfrac{1}{2},\x_P\right)= \sum_{P\in \p} Q_1\big(\lo|P|\big)\big(1+o(1)\big),
\end{equation} 
where $Q_1(x)$ is polynomial of degree $1.$ From the contour integral formula for $Q_k(x)$ in (\ref{Q cont int}), we have

\begin{equation}\label{Q_1}
Q_1(x)= \frac{1}{\pi i} \oint \frac{G(z_1)\Delta(z_1^2)^2}{z_1} \, q^{\frac{x}{2}z_1} \, dz_1,
\end{equation} 
where 

\begin{equation}
G(z_1) = A\left(\tfrac{1}{2};z_1\right) \xx\left(\tfrac{1}{2}+z_1\right)^{-\frac{1}{2}}\z(1+2z_1).
\end{equation}
\newline

Recall that, the Vandermonde determinant is defined

\begin{equation}
\Delta(z_1,\cdots,z_k) = \prod_{1\le i< j\le k}(z_j-z_i)
\end{equation}
which for $k=1$ is equal to  

\begin{equation}
\Delta(z_1^2)^2=1,
\end{equation}
and

\begin{equation}
\xx\left(\tfrac{1}{2}+z_1\right)^{-\frac{1}{2}}=q^{-z_1/2}.
\end{equation}
\newline

Therefore, (\ref{Q_1}) becomes

\begin{equation}\label{Q_12}
\begin{split}
Q_1(x)=\frac{1}{\pi i} \oint \frac{A\left(\tfrac{1}{2};z_1\right) \z(1+2z_1)} {z_1} \, q^{\frac{x-1}{2}z_1} \, dz_1,
\end{split}
\end{equation}
with

\begin{equation}
\begin{split}
A\left(\tfrac{1}{2};z_1\right)&= \prod_{\substack{P \text{ monic} \\ \text{irreducible}}} \left(1-\frac{1}{|P|^{1+2z_1}}\right)\\
& \times \frac{1}{2} \left(\left(1-\frac{1}{|P|^{1/2+z_1}}\right)^{-1}+\left(1+\frac{1}{|P|^{1/2+z_1}}\right)^{-1}\right).
\end{split}
\end{equation}
\newline

In order to compute the integral in (\ref{Q_12}) where the contour is a small circle around the origin, we need to locate the poles of the integrand. So let 

\begin{equation}\label{f(z_1)}
f(z_1)=\frac{A\left(\tfrac{1}{2};z_1\right) \z(1+2z_1)}{z_1} \, q^{\frac{x-1}{2}z_1},
\end{equation}
note that the zeta function $\z(1+2z_1)$ has a simple pole at $z_1=0,$ which means that $f(z_1)$ has a pole of order $2$ at $z_1=0.$ we compute the residue by expand $f(z_1)$ as a Laurent series and consider the coefficient of $z_1^{-1}.$ Expanding the numerator of $f(z_1)$ around $z_1=0$ we have, 

\begin{enumerate}
	\item 
	$$A\left(\tfrac{1}{2};z_1\right)= A\left(\tfrac{1}{2};0\right) + A'\left(\tfrac{1}{2};0\right)z_1+ \frac{1}{2}A''\left(\tfrac{1}{2};0\right)z_1^2+\cdots$$
	
	\item 
	$$\z(1+2z_1) = \frac{1}{2\log q}\frac{1}{z_1} +\frac{1}{2}+\frac{1}{6} (\log q) z_1 -\frac{1}{90} (\log q)^3 z_1^3 + \cdots$$
	
%

    \item 
    $$q^{\frac{x-1}{2}z_1}= 1+\frac{1}{2} (x-1) (\log q) z_1 +\frac{1}{8} (x-1)^2 (\log q)^2 z_1^2+ \cdots$$
    
\end{enumerate} 

Hence, $f(z_1)$ can be written as

\begin{equation}
\begin{split}
& f(z_1) \\
= & \Big( A(\tfrac{1}{2};0)\frac{1}{z_1} + A'(\tfrac{1}{2};0)+ \frac{1}{2}A''(\tfrac{1}{2};0)z_1+\cdots\Big)\\
& \times \Big(\frac{1}{2\log q}\frac{1}{z_1} +\frac{1}{2}+\frac{1}{6} (\log q) z_1 -\frac{1}{90} (\log q)^3 z_1^3 + \cdots\Big)\\
& \times \Big( 1 -\frac{1}{2} (\log q) z_1+ \frac{1}{8} (\log q)^2 z_1^2+\cdots\Big)\\
& \times \Big(1 +\frac{1}{2} (\log q)x z_1+ \frac{1}{8} (\log q)^2 x^2 z_1^2+\cdots\Big).\\
\end{split}
\end{equation}
Considering the coefficient of $z_1^{-1}$ we have

\begin{equation}
\text{Res}_{z_1=0} f(z_1)=  \frac{1}{4}(1+x) A(\tfrac{1}{2};0) + \frac{1}{2\log q} A'(\tfrac{1}{2};z_1).
\end{equation}
\newline

After straightforward calculations, using the definition for $A_k\left(\tfrac{1}{2},z_1,\cdots,z_k\right)$, we have 

\begin{equation}
\begin{split}
A\left(\tfrac{1}{2};z_1\right) = 1, \ \ \text{and} \ \ \  A'\left(\tfrac{1}{2};z_1\right) = 0,
\end{split}
\end{equation}
%
and so

\begin{equation}
\text{Res}_{z_1=0}f(z_1)=\frac{1}{4}(1+x).
\end{equation} 
\newline

Hence, we have 

\begin{equation}
\begin{split}
Q_1(x) & = \frac{1}{4\pi i}(1+x) \oint1 \, dz_1\\
&= \frac{1}{2} (1+x).
\end{split}
\end{equation}
\newline

Finally, we can write the first moment as,

\begin{equation}
\begin{split}
\sum_{P\in\p} L\left(\tfrac{1}{2},\x_P\right) & = \sum_{P\in \p} Q_1(\lo|P|) \left(1+o\left(1\right)\right)\\ 
& = \sum_{P\in \p} \frac{1}{2}\left(1+\lo|P|\right) \left(1+o\left(1\right)\right)\\
& = \frac{|P|}{2\lo|P|}  \left(1+ \lo|P|\right) +o\big(|P|\big).
\end{split}
\end{equation}
\newline

If we compare Theorem 2.4 of \cite{a&kprimemean} 
%
with the conjecture, we can see that the main term and the principal lower order terms are the same. In other words, Theorem 2.4 of \cite{a&kprimemean} proves our conjecture with an error $O\big(|P|^{3/4+\epsilon}\big).$

\subsection{Second moment}\label{second moment} $\text{\color{white}kgm}$\\

For $k=2$, the conjecture \ref{ourconjecture} gives

\begin{equation}
\sum_{P\in \p} L\left(\tfrac{1}{2},\x_P\right)^2= \sum_{P\in \p} Q_2\left(\lo|P|\right)\left(1+o(1)\right),
\end{equation} 

where $Q_2(x)$ is a polynomial of degree $3$, given by 

\begin{equation}\label{Q_2}
Q_2(x)= \frac{-1}{2\pi^2} \oint \oint \frac{G(z_1,z_2)\Delta(z_1^2,z_2^2)^2}{z_1^3z_2^3} \, q^{\frac{x}{2}(z_1+z_2)} \, dz_1dz_2,
\end{equation} 
with 

\begin{equation}
\begin{split}
G(z_1,z_2)  
= A&\left(\tfrac{1}{2};z_1,z_2\right) \xx\left(\tfrac{1}{2}+z_1\right)^{-\frac{1}{2}}  \xx\left(\tfrac{1}{2}+z_2\right)^{-\frac{1}{2}} \\
& \text{\color{white}kg}\times \z(1+2z_1)\z(1+z_1+z_2)\z(1+2z_2),
\end{split}
\end{equation}

\begin{equation}
\xx\left(\tfrac{1}{2}+z_1\right)^{-\frac{1}{2}}\xx\left(\tfrac{1}{2}+z_2\right)^{-\frac{1}{2}} =q^{-\frac{1}{2}(z_1+z_2)},
\end{equation}
and

\begin{equation}
\begin{split}
\Delta(z_1^2,z_2^2)^2  =(z_2^2-z_1^2)^2.
\end{split}
\end{equation}
\newline

If 
\begin{equation}
\begin{split}
&f(z_1,z_2)\\
& =\frac{A\left(\tfrac{1}{2};z_1,z_2\right) \z(1+2z_1)\z(1+z_1+z_2)\z(1+2z_2)(z_2^2-z_1^2)^2} {z_1^3z_2^3}   q^{\frac{x-1}{2}(z_1+z_2)},
\end{split}
\end{equation}
then we have 

\begin{equation}\label{Q_22}
\begin{split}
Q_2(x)=&\frac{-1}{2\pi^2} \oint \oint f(z_1,z_2)  \, dz_1dz_2\\
= & \frac{1}{24 \log ^3(q)} \Big((x^3+6x^2+11x+6) A(1/2;0,0) \log^3(q) + (3x^2+12x+11)\\
&  \log^2(q) (A_1(\tfrac{1}{2};0,0)+A_2(\tfrac{1}{2};0,0)) 12(2+x) \log(q) A_{12}(\tfrac{1}{2};0,0) \\
& - 2 (A_{222}(\tfrac{1}{2};0,0) -3A_{122}(\tfrac{1}{2};0,0) -3A_{112}(\tfrac{1}{2};0,0) +A_{111}(\tfrac{1}{2};0,0))\Big),
\end{split}
\end{equation}
where $A_j$ is the partial derivative, evaluate at zero, of the function $A\left(\tfrac{1}{2};z_1,\cdots,z_k\right)$ with respect to $j$th variable, with indices denoting higher derivatives, i.e:

$$A_{122}\left(\tfrac{1}{2};0,\cdots,0\right)=\frac{\partial}{\partial z_1}\frac{\partial^2}{\partial z_2^2}A\left(\tfrac{1}{2};z_1,\cdots,z_k\right) \Bigg|_{z_1=z_2=\cdots=z_k=0}.$$
\newline


Hence we can write the leading order asymptotic for the second moment for the family of $L$-function when $g\to\infty$ as 

\begin{equation}\label{L^2} 
\begin{split}
\sum_{P\in\p}L(\tfrac{1}{2},\x_P)^2 &\sim \sum_{P\in\p} \frac{1}{24}  \lo^3|P| A(1/2;0,0)\\
& = \frac{1}{24 \z(2)} |P| (\lo|P|)^2.
\end{split}
\end{equation}
\newline

Comparing with Andrade and Keating result (Theorem 2.5 of \cite{a&kprimemean}) we see that their theorem proves our conjecture with an error $O\left(|P| \lo|P|\right).$

\subsection{Third moment}$\text{\color{white}kgm}$\\

 For the third moment, Conjecture \ref{ourconjecture} states that

\begin{equation}
\sum_{P\in \p} L\left(\tfrac{1}{2},\x_P\right)^3= \sum_{P\in \p} Q_3\big(\lo|P|\big)\big(1+o(1)\big),
\end{equation} 

where $Q_3(x)$ is a polynomial of degree $3$. 
\newline

Thus, with the help of the symbolic manipulation software Mathematica we compute the triple contour integral and obtain 

\begin{equation}
\begin{split}
&Q_3(x)\\
=&\frac{1}{8640 \log ^6(q)}\Bigg(3 (x+3)^2 \left(x^4+12 x^3+49 x^2+78 x+40\right) A(0,0,0) \log ^6(q)\\
& +4 \left(3 x^5+45 x^4+260 x^3+720 x^2+949 x+471\right) \Big(A_{3}(0,0,0)+A_{2}(0,0,0)\\
& +A_{1}(0,0,0)\Big) \log ^5(q)+4 \left(15 x^4+180 x^3+780 x^2+1440 x+949\right) \Big(A_{23}(0,0,0)\\
& +A_{13}(0,0,0)+A_{12}(0,0,0)\Big) \log ^4(q)-10 \left(x^3+9 x^2+26 x+24\right) \Big(2 A_{333}(0,0,0)\\
& -3 A_{233}(0,0,0)-3 A_{223}(0,0,0)+2 A_{222}(0,0,0)-3 A_{133}(0,0,0)-36 A_{123}(0,0,0)\\
& -3 A_{122}(0,0,0)-3 A_{113}(0,0,0)-3 A_{112}(0,0,0)+2 A_{111}(0,0,0)\Big) \log ^3(q)\\
& -20 \left(3 x^2+18 x+26\right) \Big(A_{2333}(0,0,0)+A_{2223}(0,0,0)+A_{1333}(0,0,0)-6 A_{1233}(0,0,0)\\
& -6 A_{1223}(0,0,0)+A_{1222}(0,0,0)-6 A_{1123}(0,0,0)+A_{1113}(0,0,0)+A_{1112}(0,0,0)\Big) \\
&\log ^2(q)+6 (x+3) \Big(2 A_{33333}(0,0,0)-5 A_{23333}(0,0,0)-10 A_{22333}(0,0,0)-10 A_{22233}(0,0,0)\\
& -5 A_{22223}(0,0,0)+2 A_{22222}(0,0,0)-5 A_{13333}(0,0,0)+60 A_{12233}(0,0,0)-5 A_{12222}(0,0,0)\\
& -10 A_{11333}(0,0,0)+60 A_{11233}(0,0,0)+60 A_{11223}(0,0,0)-10 A_{11222}(0,0,0)\\
& -10 A_{11133}(0,0,0)-10 A_{11122}(0,0,0)-5 A_{11113}(0,0,0)-5 A_{11112}(0,0,0)\\
& +2 A_{11111}(0,0,0)\Big) \log (q)+4 \Big(3 A_{233333}(0,0,0)-20 A_{222333}(0,0,0)+3 A_{222223}(0,0,0)\\
& +3 A_{133333}(0,0,0)-30 A_{123333}(0,0,0)+30 A_{122333}(0,0,0)+30 A_{122233}(0,0,0)\\
& -30 A_{122223}(0,0,0)+3 A_{122222}(0,0,0)+30 A_{112333}(0,0,0)+30 A_{112223}(0,0,0)\\
& -20 A_{111333}(0,0,0)+30 A_{111233}(0,0,0)+30 A_{111223}(0,0,0)-20 A_{111222}(0,0,0)\\
& -30 A_{111123}(0,0,0)+3 A_{111113}(0,0,0)+3 A_{111112}(0,0,0)\Big)\Bigg),
\end{split}
\end{equation}
where $A\left(\tfrac{1}{2};z_1,z_2,z_3\right)$ is defined in Lemma \ref{A_k}. Hence the leading order asymptotic for the third moment for our family of $L$-functions is given by

\begin{equation}\label{L^3} 
\begin{split}
\sum_{P\in\p}L\left(\tfrac{1}{2},\x_P\right)^3 &\sim \sum_{P\in \p} \frac{1}{2880} A\left(\tfrac{1}{2};0,0,0\right) (\lo|P|)^6 \\
& = \frac{1}{2880} |P| A(\tfrac{1}{2};0,0,0) (\lo|P|)^5,
\end{split}
\end{equation}
where
 
\begin{equation}
\begin{split}
A(\tfrac{1}{2};0,0,0) = \prod_{\substack{P \text{ monic} \\ \text{irreducible}}}  \left(1- \frac{6 |P|^2-8|P|+3}{|P|^4}\right).\\
\end{split}
\end{equation}

\subsection{Leading order for general $k$}$\text{\color{white}kgm}$\\

The main aim in this section is to obtain a conjecture for the leading order asymptotic of the moments for a general integer $k$. The calculations presented here are based to the calculations first presented in \cite{Kea Odg} and \cite{andradePhD}. To obtain the main formula we need the following lemma.

%
%

\begin{lemma}
	Let $F$ be a symmetric function of $k$ variables, regular near $(0,\cdots,0)$ and $f(s)$ has a simple pole of residue $1$ at $s=0$ and analytic in a neighbourhood of $s=0.$ Let 
	\begin{equation}
	\begin{split}
	K\left(|P|;w_1,\cdots,w_k\right) = \sum_{\varepsilon_i=\pm 1}e^{\frac{1}{2} \log|P| \sum_{i=1}^k \varepsilon_i w_i} & F\left(\varepsilon_1 w_1,\cdots,\varepsilon_k w_k\right) \\
	& \times \prod_{1 \leq i \leq j \leq k} f\left(\varepsilon_i w_i + \varepsilon_j w_j\right),
	\end{split}
	\end{equation}
	and define $I\left(|P|,k;w=0\right)$ to be the value of $K$ when $w_1,\cdots,w_k=0.$ We have that,
	\begin{equation}
	I\left(|P|,k;0\right) \sim \left(\frac{1}{2}\log|P|\right)^{k(k+1)/2}F(0,\cdots,0) 2^{k(k+1)/2}. \left(\prod_{i=1}^k \frac{i!}{\left(2i\right)!}\right)
	\end{equation}
\end{lemma} 
\begin{proof} See Lemma 5 in \cite{a&kConInMo}.
\end{proof}

We are in a position to obtain the desired formula, from (\ref{6.4.9}) recall that

\begin{equation}
\begin{split}
&\sum_{P\in \p} \prod_{i=1}^k L(\tfrac{1}{2}+\alpha_i,\x_P)\color{white}\text{dngkm  n  dncmkmlm     dnlkmmnlkfmflkn  lgkmbfn b lkdmnb kmb nbmcfkn}\\
&=\sum_{P\in \p} \prod_{i=1}^k\frac{|P|^{-\frac{1}{2}\sum_{i=0}^k \alpha_i} \xx(\tfrac{1}{2}+\alpha_i)^{\frac{1}{2}}} {(\log q)^{k(k+1)/2}} \sum_{\varepsilon_i=\pm 1} K(\varepsilon_1\alpha_1,\cdots,\varepsilon_k\alpha_k)\Big(1+o\big(1\big)\Big),\\
\end{split}
\end{equation}
where 

\begin{equation}
\begin{split}
K(&\varepsilon_1\alpha_1,\cdots,\varepsilon_k\alpha_k)\\
=&\sum_{\varepsilon_i=\pm 1}  \prod_{i=1}^k \xx(\tfrac{1}{2}+\varepsilon_i\alpha_i)^{-\frac{1}{2}} A_k(\tfrac{1}{2};\alpha_1,\cdots,\alpha_k) |P|^{\frac{1}{2}\sum_{i=0}^k \varepsilon_i \alpha_i} \\
& \text{\color{white} dghhgd} \times \prod_{1\le i< j\le k}\z(1+\varepsilon_i\alpha_i+\varepsilon_j\alpha_j)(\log q).
\end{split}
\end{equation}
\newline

Applying the above Lemma with

\begin{equation*}
\begin{split}
f(s)&= \z(1+s) \log q,\\
F\left( w_1,\cdots, w_k\right) &= \prod_{i=1}^k \xx(\tfrac{1}{2}+\alpha_i)^{-\frac{1}{2}} A_k\left(\tfrac{1}{2};w_1,\cdots, w_k\right),\\
K\left(|P|;w_1,\cdots,w_k\right) &= \sum_{\varepsilon_i=\pm 1}|P|^{\frac{1}{2}  \sum_{i=1}^k \varepsilon_i w_i}  F\left(\varepsilon_1 w_1,\cdots,\varepsilon_k w_k\right)\\
& \text{\color{white} dljht} \times \prod_{1 \leq i \leq j \leq k} f\left(\varepsilon_i w_i + \varepsilon_j w_j\right),
\end{split}
\end{equation*}
and letting $\alpha_1,\cdots,\alpha_k \to 0$ we obtain 

\begin{equation}
\begin{split}
\sum_{P\in \p}  L(\tfrac{1}{2},\x_P)^k \sim & \sum_{P\in \p} \frac{1}{(\log q)^{k(k+1)/2}} \left(\frac{1}{2} \log|P|\right)^{\frac{k(k+1)}{2}}\\
& \times A(\tfrac{1}{2};0,\cdots,0) 2^{\frac{k(k+1)}{2}} \prod_{i=1}^{k} \frac{i!}{\left(2i\right)!},\\
\end{split}
\end{equation}
as $g\to \infty$. Summing over $P$ we get that

\begin{equation}
\begin{split}
\sum_{P\in\p} L(\tfrac{1}{2};\x_P)^k &\sim \sum_{P\in\p} \left(\lo|P|\right)^{\frac{k(k+1)}{2}} A_k(\tfrac{1}{2};0,\cdots,0) \prod_{i=1}^k\frac{i!}{\left(2i\right)!}\\
&= |P| \left(\lo|P|\right)^{\frac{k(k+1)}{2}-1} A_k(\tfrac{1}{2};0,\cdots,0) \prod_{i=1}^k\frac{i!}{\left(2i\right)!}.
\end{split}
\end{equation}
\newline

Hence, we have proved the following.

\begin{theorem}\label{leading order}
	Conditional on Conjecture \ref{ourconjecture} we have that as $g\to \infty$ the following holds
	
	\begin{equation}
	\begin{split}
	\sum_{P\in \p}  L(\tfrac{1}{2},\x_P)^k \sim |P|  \left( \lo|P|\right)^{\frac{k(k+1)}{2}-1} A(\tfrac{1}{2};0,\cdots,0)  \prod_{i=1}^{k} \frac{i!}{\left(2i\right)!},\\
	\end{split}
	\end{equation}
\end{theorem}

\subsubsection{Some Conjectural Values for Leading Order Asymptotic for the Moments of $L(s,\x_P)$}$\text{\color{white}kgm}$\\

We end this section by writing the asymptotic formula for the fourth and the fifth moment for our family of $L$-functions. Theorem \ref{leading order} implies that the leading order for the fourth moment can be written as 

\begin{equation}
\begin{split}
\sum_{P\in\p} L\left(\tfrac{1}{2},\x_P\right)^4 &\sim |P| \left(\lo|P|\right)^{9}A\left(\tfrac{1}{2};0,0,0,0\right) \prod_{i=1}^4\frac{i!}{\left(2i\right)!}\\
&= \frac{1}{4838400}|P| \left(\lo|P|\right)^{9}A\left(\tfrac{1}{2};0,0,0,0\right),
\end{split}
\end{equation}
where

\begin{equation*}
\begin{split}
A&\left(\tfrac{1}{2};0,0,0,0\right)\\
& = \prod_{\substack{P \text{ monic} \\ \text{irreducible}}} \left(1- \frac{20 |P|^6-64 |P|^5+90 |P|^4-64 |P|^3+20 |P|^2-1}{|P|^{8}}\right),
\end{split}
\end{equation*} 
and the leading order for the fifth moment is 

\begin{equation}
\begin{split}
\sum_{P\in\p} &L\left(\tfrac{1}{2},\x_P\right)^5 \\
& \sim |P| \left(\lo|P|\right)^{14}A\left(\tfrac{1}{2};0,0,0,0,0\right) \prod_{i=1}^5\frac{i!}{\left(2i\right)!}\\
&= \frac{1}{146313216000}|P| \left(\lo|P|\right)^{14}\prod_{\substack{P \text{ monic} \\ \text{irreducible}}} \left(1-\frac{h(|P|)}{|P|^{12}}\right).
\end{split}
\end{equation}
with 

\begin{equation}
\begin{split}
h(x)= &50 x^{10}-280 x^9+765 x^8-1248 x^7+1260 x^6 -720 x^5\\
& +105 x^4 +160 x^3-126 x^2+40 x-5.
\end{split}
\end{equation}

\section{Ratios conjecture for $L$-functions over function fields}\label{ratios conjecture}

The main aim of this section is to obtain a conjectural asymptotic formula for 

\begin{equation}\label{frac}
\sum_{P\in \p} \frac{\prod_{k=1}^K L(\frac{1}{2}+\alpha_k,\x_P)}{\prod_{q=1}^Q L(\frac{1}{2}+\gamma_q,\x_P)},
\end{equation}
where $\p=\{P \text{ monic, } P \text{ irreducible, deg}(P)=2g+1, P\in\f\},$ and $\mathfrak{P}=\{L(s,\x_P):P\in\p\}$. We adapt the original recipe of Conrey, Farmer and Zirnbauer \cite{Conr-Far-Zir} for this family of $L$-functions. 
\newline

The idea is to replace the $L$-functions in the numerator by their ``approximate" functional equation

\begin{equation}\label{paproxi}
L(s,\x_P)=\sum_{\substack{n \text{ monic}\\\text{deg}(n)\le g}} \frac{\x_P(n)}{|n|^s}+  \xx_P(s) \sum_{\substack{n \text{ monic}\\\text{deg}(n)\le g-1}} \frac{\x_P(n)}{|n|^{1-s}},
\end{equation}
and expand the $L$-functions in the denominator into the series

\begin{equation}\label{series}
\begin{split}
\frac{1}{L(s,\x_P)} =& \prod_{\substack{P \text{ monic} \\ \text{irreducible}}}\left(1-\frac{\x_P(P)}{|P|^s}\right) \\
=& \sum_{n \text{ monic}} \frac{\mu(n) \x_P(n)}{|n|^s},
\end{split}
\end{equation}
where $\mu(n)$ and $\x_P(n)$ is defined in Section \ref{backgrownd}.
\newline

As in the previous section, we apply the recipe to the quantity

\begin{equation}
\sum_{P\in \p} \frac{\prod_{k=1}^K Z_\lL(\frac{1}{2}+\alpha_k,\x_P)}{\prod_{q=1}^Q L(\frac{1}{2}+\gamma_q,\x_P)}
\end{equation} 
where $Z_{\lL}(s,\x_P)$ is defined in (\ref{Z_L}) with ``approximate" functional equation given by (\ref{prime_xapro}). Now expanding the denominator we get 

\begin{equation}\label{Z-frac}
\begin{split}
\sum_{P\in \p} & \frac{\prod_{k=1}^K Z_\lL(\frac{1}{2}+\alpha_k,\x_P)}{\prod_{q=1}^Q L(\frac{1}{2}+\gamma_q,\x_P)}\\
= & \sum_{P\in \p} \prod_{k=1}^K Z_\lL(\tfrac{1}{2}+\alpha_k,\x_P)  \sum_{\substack{h_1,\cdots,h_Q \\h_q \text{ monic}}} \frac{\mu(h_1)\cdots \mu(h_Q) \x_P(h_1 \cdots h_Q)}{|h_1|^{\frac{1}{2}+\gamma_1} \cdots |h_Q|^{\frac{1}{2}+\gamma_Q}}.
\end{split}
\end{equation} 
\newline

Making use of the ``approximate" functional equation (\ref{paproxi}), we have 

\begin{equation}
\begin{split}
\prod_{k=1}^K &Z_\lL(\tfrac{1}{2}+\alpha_k,\x_P)\\
=& \sum_{\varepsilon_k\in\{-1,1\}^K} \prod_{k=1}^K \xx_P(\tfrac{1}{2}+\varepsilon_k\alpha_k)^{-\frac{1}{2}} \sum_{\substack{m_1,\cdots,m_K \\ m_i \text{ monic}}} \frac{\x_P(m_1 \cdots m_K)}{|m_1|^{\frac{1}{2}+\varepsilon_1\alpha_1} \cdots |m_K|^{\frac{1}{2}+\varepsilon_K\alpha_K}},
\end{split}
\end{equation}
so we can write (\ref{Z-frac}) as

\begin{equation}\label{Z}
\begin{split}
\sum_{P\in \p} & \frac{\prod_{k=1}^K Z_\lL(\frac{1}{2}+\alpha_k,\x_P)}{\prod_{q=1}^Q L(\frac{1}{2}+\gamma_q,\x_P)}\\
= & \sum_{P\in \p} \sum_{\varepsilon_k\in\{-1,1\}^K} \prod_{k=1}^K \xx_P(\tfrac{1}{2}+\varepsilon_k\alpha_k)^{-\frac{1}{2}}\\
& \color{white}\text{dljnkng}\color{black} \times \sum_{\substack{m_1,\cdots,m_K \\ h_1,\cdots,h_Q \\m_i,h_j \text{ monic}}} \frac{\prod_{q=1}^Q \mu(h_q) \x_P(\prod_{k=1}^K m_k \prod_{q=1}^Q h_q)}{\prod_{k=1}^K |m_k|^{\frac{1}{2}+\varepsilon_k\alpha_k} \prod_{q=1}^Q |h_q|^{\frac{1}{2}+\gamma_q}}.\\
\end{split}
\end{equation} 
\newline

Following the recipe we replace each summand by its expected value when averaged over primes $P\in\p$, in other words we have that

\begin{equation*}
\begin{split}
\lim_{\text{deg}(P)\to\infty} \Bigg(\frac{1}{\#\p} & \sum_{P\in \p} \sum_{\varepsilon_k\in\{-1,1\}^K} \prod_{k=1}^K \xx_P(\tfrac{1}{2}+\varepsilon_k\alpha_k)^{-\frac{1}{2}}\\
&  \times \sum_{\substack{m_1,\cdots,m_K \\ h_1,\cdots,h_Q \\m_i,h_j \text{ monic}}} \frac{\prod_{q=1}^Q \mu(h_q) \x_P(\prod_{k=1}^K m_k \prod_{q=1}^Q h_q)}{\prod_{k=1}^K |m_k|^{\frac{1}{2}+\varepsilon_k\alpha_k} \prod_{q=1}^Q |h_q|^{\frac{1}{2}+\gamma_q}}\Bigg)
\end{split}
\end{equation*} 
\begin{equation}
\begin{split}
= \sum_{\varepsilon_k\in\{-1,1\}^K}  \prod_{k=1}^K \xx_P&(\tfrac{1}{2}+\varepsilon_k\alpha_k)^{-\frac{1}{2}}\\
&\times \sum_{\substack{m_1,\cdots,m_K \\ h_1,\cdots,h_Q \\m_i,h_j \text{ monic}}} \frac{\prod_{q=1}^Q \mu(h_q) \delta\left(\prod_{k=1}^K m_k \prod_{q=1}^Q h_q\right)}{\prod_{k=1}^K |m_k|^{\frac{1}{2}+\varepsilon_k\alpha_k} \prod_{q=1}^Q |h_q|^{\frac{1}{2}+\gamma_q}},
\end{split}
\end{equation} 
where $\delta(n)=1$ if $n$ is a square and $0$ otherwise.
\newline

Next we factor out the zeta-function factors. Note that, the main difficulty here is to identify and factor out the appropriate zeta-functions factors that contribute to poles and zeros. With the same notation used in \cite{andradePhD}, we define the following series

\begin{equation}\label{GG}
G_{\mathfrak{P}} (\alpha;\gamma) = \sum_{\substack{m_1,\cdots,m_K \\ h_1,\cdots,h_Q \\m_i,h_j \text{ monic}}} \frac{\prod_{q=1}^Q \mu(h_q) \delta\left(\prod_{k=1}^K m_k \prod_{q=1}^Q h_q\right)}{\prod_{k=1}^K |m_k|^{\frac{1}{2}+\varepsilon_k\alpha_k} \prod_{q=1}^Q |h_q|^{\frac{1}{2}+\gamma_q}}.
\end{equation}
If $m_k=\prod_P P^{a_k}$ and $h_q=\prod_P P^{c_q}$, then we can write $G_{\mathfrak{P}}(\alpha;\gamma)$ as a convergent Euler product provided that $\mathfrak{R}(\alpha_k)>0$ and $\mathfrak{R}(\gamma_q)>0$,  

\begin{equation}\label{wG}
\begin{split}
G&_{\mathfrak{P}} (\alpha;\gamma)\\ 
&=  \prod_{\substack{P \text{ monic}\\ \text{irreducible}}}\left(1+ \sum_{0<\sum_ka_k+\sum_qc_q \text{ is even}} \frac{\prod_{q=1}^Q \mu(P^{c_q})}{ |P|^{\sum_ka_k(\frac{1}{2}+\alpha_k)+\sum_qc_q(\frac{1}{2}+\gamma_q)}}\right).
\end{split}
\end{equation}
\newline

We now write $G_{\mathfrak{P}}$ in terms of the zeta-function of $\f$. First, we express the contribution of all poles and zeros of (\ref{wG}) in terms of $\z(s)$ by rewriting the Euler product in (\ref{wG}) as

\begin{equation}\label{wG1}
\begin{split}
G&_{\mathfrak{P}} (\alpha;\gamma)\text{\color{white}gkfhoijgoijoibjolioljoljmolhjmoirnjjijfdrjhbijijbrjjbjnfdbn ijg}\\
& =  \prod_{\substack{P \text{ monic}\\ \text{irreducible}}}\Bigg(1+ \sum_{\substack{j,k\\ j<k}} \frac{1}{ |P|^{(\frac{1}{2}+\alpha_j)+(\frac{1}{2}+\alpha_k)}}+ \sum_{k} \frac{1}{ |P|^{(1+2\alpha_k)}}\\
&\text{\color{white}gkf} + \sum_{\substack{r,q\\ r<q}} \frac{\mu(P)^2}{ |P|^{(\frac{1}{2}+\gamma_r)+(\frac{1}{2}+\gamma_q)}}+ \sum_{k}\sum_q \frac{\mu(P)}{ |P|^{(\frac{1}{2}+\alpha_k)+(\frac{1}{2}+\gamma_q)}}+ \cdots\Bigg),
\end{split}
\end{equation}  
where $\cdots$ are referring to the convergent terms. Recall that

\begin{equation}
\begin{split}
\z(s)&= \prod_{\substack{P \text{ monic} \\ \text{irreducible}}} \left(1-\frac{1}{|P|^s}\right)^{-1}\\
&= \prod_{\substack{P \text{ monic} \\ \text{irreducible}}} \left(\sum_{j=0}^\infty \left(\frac{1}{|P|^s}\right)^j\right). 
\end{split}
\end{equation}
\newline

We can see from (\ref{wG1}) that the terms with $\sum_{k=1}^Ka_k+\sum_{q=1}^Q c_q=2$ contribute to the poles and zeros. The poles are coming from the terms with $a_j=a_k=1,1\le j<k\le K,$ $a_k=2, 1\le k\le K,$ and also from the terms with $c_r=c_q=1, 1\le r < q\le Q$. Note that there are no poles coming from the terms with $c_q=2, 1\le q \le Q,$ since $\mu(P^2)=0.$ Moreover, the zeros comes from the terms with $a_k=c_q=1$ with $1\le k\le K,$ and $1\le q\le Q.$ 
\newline

From the above, we can define the function $Y_S(\alpha;\gamma)$ in terms of $\z(s)$ by,

\begin{equation}
Y_S(\alpha;\gamma):= \frac{\prod_{1\le j\le k\le K}\z(1+\alpha_j+\alpha_k)\prod_{1\le r\le q\le Q}\z(1+\gamma_r+\gamma_q)}{\prod_{k=1}^k\prod_{q=1}^Q \z(1+\alpha_k+\gamma_q)}.
\end{equation}
Thus, we can factor out $Y_S(\alpha;\gamma)$ from $G_{\mathfrak{P}}(\alpha;\gamma)$, such that 

\begin{equation}\label{GYA}
G_{\mathfrak{P}}(\alpha;\gamma)= Y_S(\alpha;\gamma) A_{\mathfrak{P}}(\alpha;\gamma),
\end{equation}where $A_{\mathfrak{P}}(\alpha;\gamma)$ is the Euler product that converge absolutely for all of the variables in the small disks around $0,$

\begin{equation}\label{A_P}
\begin{split}
&A_{\mathfrak{P}}(\alpha;\gamma)\\
&= \prod_{\substack{P \text{ monic}\\ \text{irreducible}}} \frac{\prod_{1\le j\le k\le K} \left(1-\frac{1}{|P|^{1+\alpha_j+\alpha_k}}\right) \prod_{1\le r\le q\le Q} \left(1-\frac{1}{|P|^{1+\gamma_r+\gamma_q}}\right)}{\prod_{k=1}^k\prod_{q=1}^Q \left(1-\frac{1}{|P|^{1+\alpha_k+\gamma_q}}\right)}\\
&\color{white} \text{fdmj gf jf ghf} \color{black} \times \left(1+ \sum_{0<\sum_ka_k+\sum_qc_q \text{ is even}} \frac{\prod_{q=1}^Q \mu(P^{c_q})}{ |P|^{\sum_ka_k(\frac{1}{2}+\alpha_k)+\sum_qc_q(\frac{1}{2}+\gamma_q)}}\right).
\end{split}
\end{equation}
\newline

Returning to the recipe, we can conclude from (\ref{Z}), (\ref{GG}), and (\ref{GYA}) that 

\begin{equation}
\begin{split}
&\sum_{P\in \p} \frac{\prod_{k=1}^K Z_\lL(\frac{1}{2}+\alpha_k,\x_P)}{\prod_{q=1}^Q L(\frac{1}{2}+\gamma_q,\x_P)}\\
&\text{\color{white}gkjj}= \sum_{P\in \p} \sum_{\varepsilon_k\in\{-1,1\}^K} \prod_{k=1}^K \xx_P\left(\tfrac{1}{2}+\varepsilon_k\alpha_k\right)^{-\frac{1}{2}} Y_S(\varepsilon_1\alpha_1,\cdots ,\varepsilon_k\alpha_k;\gamma)\\
& \text{\color{white}gkgf;ljj} \times A_{\mathfrak{P}}(\varepsilon_1\alpha_1,\cdots ,\varepsilon_k\alpha_k;\gamma) +o\left(|P|\right),\\
\end{split}
\end{equation} 
Now, using (\ref{Z_L}) we have

\begin{equation}
\begin{split}
\sum_{P\in \p} & \frac{\prod_{k=1}^K L(\frac{1}{2}+\alpha_k,\x_P)}{\prod_{q=1}^Q L(\frac{1}{2}+\gamma_q,\x_P)}\\
= & \sum_{P\in \p} \sum_{\varepsilon_k\in\{-1,1\}^K} \prod_{k=1}^K \xx_P\left(\tfrac{1}{2}+\alpha_k\right)^{\frac{1}{2}} \xx_P\left(\tfrac{1}{2}+\alpha_k\right)^{-\frac{1}{2}}\\
& 
\times Y_S(\varepsilon_1\alpha_1,\cdots ,\varepsilon_k\alpha_k;\gamma) A_{\mathfrak{P}}(\varepsilon_1\alpha_1,\cdots ,\varepsilon_k\alpha_k;\gamma) +o\left(|P|\right).\\
\end{split}
\end{equation}
\newline

Remembering that,

\begin{equation}
\xx_P(s) = |P|^{\frac{1}{2}-s} \xx(s)
\end{equation}
with
\begin{equation}
\xx(s)= q^{-\frac{1}{2}+s},
\end{equation}
%
we have that

\begin{equation}
\begin{split}
\prod_{k=1}^K & \xx_P\left(\tfrac{1}{2}+\varepsilon_k\alpha_k\right)^{-\frac{1}{2}} \xx_P\left(\tfrac{1}{2}+\varepsilon_k\alpha_k\right)^{\frac{1}{2}}\\
&\color{white}\text{kfdvjyggk} \color{black} = |P|^{\frac{1}{2}\sum_{k=1}^K (\varepsilon_k\alpha_k-\alpha_k)} \prod_{k=1}^K \xx\left(\tfrac{1}{2}+\tfrac{\alpha_k-\varepsilon_k\alpha_k}{2}\right). \\
\end{split}
\end{equation}
\newline

For positive real parts of $\alpha_k$ and $\gamma_q$ we have 
 
 \begin{equation}
 \begin{split}
 &\sum_{P\in \p} \frac{\prod_{k=1}^K L(\frac{1}{2}+\alpha_k,\x_P)}{\prod_{q=1}^Q L(\frac{1}{2}+\gamma_q,\x_P)}\\
 & = \sum_{P\in \p} \sum_{\varepsilon_k\in\{-1,1\}^K}|P|^{\frac{1}{2}\sum_{k=1}^K (\varepsilon_k\alpha_k-\alpha_k)} \prod_{k=1}^K \xx\left(\tfrac{1}{2}+\tfrac{\alpha_k-\varepsilon_k\alpha_k}{2}\right)\\
 & 
 \times Y_S(\varepsilon_1\alpha_1,\cdots ,\varepsilon_k\alpha_k;\gamma) A_{\mathfrak{P}}(\varepsilon_1\alpha_1,\cdots ,\varepsilon_k\alpha_k;\gamma) +o\left(|P|\right).\\
 \end{split}
 \end{equation}
 \newline

 Finally, if we let
 
 	\begin{equation}
 \begin{split}
 H_{\mathfrak{P},|P|,\alpha,\gamma}(w;\gamma) = & \left|P\right|^{\frac{1}{2}\sum_{k=1}^k w_k} \prod_{k=1}^K \xx\left(\tfrac{1}{2}+\tfrac{\alpha_k-w_k}{2}\right)\\
 & \times Y_S(w;\gamma) A_{\mathfrak{P}}(w;\gamma),
 \end{split}
 \end{equation}
 then the conjecture may be formulated as
 
 \begin{equation}
 \begin{split}
 \sum_{P\in\p}& \frac{\prod_{k=1}^KL(\frac{1}{2}+\alpha_k,\x_P)}{\prod_{q=1}^QL(\frac{1}{2}+\gamma_q,\x_P)}\\
 & =\sum_{P\in\p} |P|^{-\frac{1}{2}\sum_{k=1}^K \alpha_k} \sum_{\varepsilon\in\{-1,1\}^K}  H_{\mathfrak{P},|P|,\alpha,\gamma}(\varepsilon\alpha;\gamma)+o\left(|P|\right).\\
 \end{split}
 \end{equation}

\subsection{Refinements of Conjecture}$\text{\color{white}kgm}$\\

In this section we state the final form of our ratios conjecture. In the first part we derive a closed form expression for the Euler product $A_\mathfrak{P}(\alpha;\gamma),$ and in the second part we express the combinatotial sum as a multiple integral. 


\subsubsection{Closed form expression for $A_\mathfrak{P}$}$\text{\color{white}kgm}$\\

Suppose that $f(x)=1+\sum_{n=1}^\infty u_nx^n.$ then

\begin{equation}
\begin{split}
 \sum_{n \text{ even}} u_n x^n
&=\frac{1}{2} \left(f(x)+f(-x)-2\right),
\end{split}
\end{equation}
and so, let

\begin{equation}
\begin{split}
f\left(\frac{1}{|P|}\right)=&\sum_{a_k,c_q} \frac{\prod_{q=1}^Q\mu(P^{c_q})}{|P|^{\sum_ka_k(\frac{1}{2}+\alpha_k)+\sum_qc_q(\frac{1}{2}+\gamma_q)}}\\
=& \sum_{a_k} \prod_{k=1}^K\frac{1}{|P|^{a_k(\frac{1}{2}+\alpha_k)}} \sum_{c_q} \prod_{q=1}^Q\frac{\mu(P^{c_q})}{|P|^{c_q(\frac{1}{2}+\gamma_q)}}\\
=& \frac{\prod_{q=1}^Q\left(1-\frac{1}{|P|^{\frac{1}{2}+\gamma_q}}\right)} {\prod_{k=1}^K\left(1-\frac{1}{|P|^{\frac{1}{2}+\alpha_k}}\right)}. \\
\end{split}
\end{equation}
\newline

Using the above equations we can establish the following lemma.

\begin{lemma}\label{a_kc_qeven}
	We have that, 
	\begin{equation}
	\begin{split}
	1+& \sum_{\sum_ka_k+\sum_qc_q \text{ even}} \frac{\prod_{q=1}^Q\mu(P^{c_q})}{|P|^{\sum_ka_k(\frac{1}{2}+\alpha_k)+\sum_qc_q(\frac{1}{2}+\gamma_q)}}\\
	&= \frac{1}{2} \left(\frac{\prod_{q=1}^Q\left(1-\frac{1}{|P|^{\frac{1}{2}+\gamma_q}}\right)} {\prod_{k=1}^K\left(1-\frac{1}{|P|^{\frac{1}{2}+\alpha_k}}\right)}+\frac{\prod_{q=1}^Q\left(1+\frac{1}{|P|^{\frac{1}{2}+\gamma_q}}\right)} {\prod_{k=1}^K\left(1+\frac{1}{|P|^{\frac{1}{2}+\alpha_k}}\right)}\right)-1.
	\end{split}
	\end{equation}
\end{lemma}

The following result is a direct corollary from Lemma \ref{a_kc_qeven} and equation (\ref{A_P}).

\begin{corollary}\label{A_P2}
	\begin{equation}
	\begin{split}
	A_{\mathfrak{P}}(\alpha;\gamma)&= \prod_{\substack{P \text{ monic}\\ \text{irreducible}}} \frac{\prod_{1\le j\le k\le K} \left(1-\frac{1}{|P|^{1+\alpha_j+\alpha_k}}\right) \prod_{1\le r\le q\le Q} \left(1-\frac{1}{|P|^{1+\gamma_r+\gamma_q}}\right)}{\prod_{k=1}^k\prod_{q=1}^Q \left(1-\frac{1}{|P|^{1+\alpha_k+\gamma_q}}\right)}\\
	& \times \left( \frac{1}{2} \left(\frac{\prod_{q=1}^Q\left(1-\frac{1}{|P|^{\frac{1}{2}+\gamma_q}}\right)} {\prod_{k=1}^K\left(1-\frac{1}{|P|^{\frac{1}{2}+\alpha_k}}\right)}+\frac{\prod_{q=1}^Q\left(1+\frac{1}{|P|^{\frac{1}{2}+\gamma_q}}\right)} {\prod_{k=1}^K\left(1+\frac{1}{|P|^{\frac{1}{2}+\alpha_k}}\right)}\right)-1\right).
	\end{split}
	\end{equation}
\end{corollary}$\text{\color{white}kgm}$

\subsubsection{The final form of the ratios conjecture}$\text{\color{white}kgm}$\\

To obtain our final form of the Ratios Conjecture \ref{Our Ratios Conjecture}, we need the following lemma (Lemma 6.8, \cite{Conr-Far-Zir}).

\begin{lemma}\label{6.8}
	Suppose that $F(z)=F(z_1,\cdots,z_K)$ is a function of $K$ variables, which is symmetric and regular near $(0,\cdots,0).$ Suppose further that $f(s)$ has a simple pole of residue $1$ at $s=0$ but is otherwise analytic in $|s|\le 1.$ Let either 
	
	\begin{equation}
	H(z_1,\cdots,z_K)=F(z_1,\cdots,z_K)\prod_{1\le j\le k\le K} f(z_j+z_k)
	\end{equation}
	or
	\begin{equation}
	H(z_1,\cdots,z_K)=F(z_1,\cdots,z_K)\prod_{1\le j< k\le K} f(z_j+z_k).
	\end{equation}
	If $|\alpha_k|<1$ then 
	
	\begin{equation}
	\begin{split}
	&\sum_{\varepsilon\in\{-1,1\}^K} H(\varepsilon_1\alpha_1,\cdots,\varepsilon_K\alpha_K)\\
	& = \frac{(-1)^{K(K-1)/2}2^K}{K! (2\pi i)^K} \int_{|z_i|=1} \frac{H(z_1,\cdots,z_K) \Delta(z_1^2,\cdots,z_K^2)^2 \prod_{k=1}^Kz_k}{\prod_{j=1}^K\prod_{k=1}^K(z_k-\alpha_j)(z_k+\alpha_j)} dz_1\cdots dz_K
	\end{split}
	\end{equation}
	
	and
\end{lemma}
\begin{equation}
\begin{split}
&\sum_{\varepsilon\in\{-1,1\}^K} \text{sgn}(\varepsilon) H(\varepsilon_1\alpha_1,\cdots,\varepsilon_K\alpha_K)\\
& = \frac{(-1)^{K(K-1)/2}2^K}{K! (2\pi i)^K} \int_{|z_i|=1} \frac{H(z_1,\cdots,z_K) \Delta(z_1^2,\cdots,z_K^2)^2 \prod_{k=1}^K\alpha_k}{\prod_{j=1}^K\prod_{k=1}^K(z_k-\alpha_j)(z_k+\alpha_j)} dz_1\cdots dz_K.
\end{split}
\end{equation}
\newline

%
%
%
Now, we are in a position to present the final form of the ratios conjecture \ref{Our Ratios Conjecture}.

\begin{conjecture}
	Suppose that the real parts of $\alpha_k$ and $\gamma_q$ are positive. Then we have,
	
	\begin{equation}
	\begin{split}
	&\sum_{P\in\p}\frac{\prod_{k=1}^KL\left(\frac{1}{2}+\alpha_k,\x_P\right)}{\prod_{q=1}^QL\left(\frac{1}{2}+\gamma_q,\x_P\right)} \\
	&=\sum_{P\in\p} |P|^{-\frac{1}{2}\sum_{k=1}^K \alpha_k} \frac{(-1)^{K(K-1)/2}2^K}{K! (2\pi i)^K}\\
	&\color{white}\text{gkj}\color{black} \times \int_{|z_i|=1} \frac{H_{\mathfrak{P},|P|,\alpha,\gamma}(z_1,\cdots,z_K) \Delta(z_1^2,\cdots,z_K^2)^2 \prod_{k=1}^Kz_k}{\prod_{j=1}^K\prod_{k=1}^K(z_k-\alpha_j)(z_k+\alpha_j)} dz_1\cdots dz_K\\
	& \color{white}\text{gkj}\color{black} +o(|P|).
	\end{split}
	\end{equation}
\end{conjecture}

\section{One-level density}\label{one level}

In this section we give an application of the Ratios Conjecture \ref{Our Ratios Conjecture} for $L$-functions over function fields. We compute a smooth linear statistic, the one-level density for the family of quadratic Dirichlet $L$-functions associated to monic irreducible polynomials in $\f$. The one-level density for the family of quadratic Dirichlet $L$-functions over fundamental discriminants was computed using the rations conjecture by Conrey and Snaith \cite{con-snaith appl int mom} in the number field setting and by Andrade and Keating \cite{a&kConInMo} in the function field setting. 
\newline

Consider 

\begin{equation}
R_P(\alpha;\gamma)= \sum_{P\in \p} \frac{L(\frac{1}{2}+\alpha,\x_P)}{L(\frac{1}{2}+\gamma,\x_P)}.
\end{equation}
Using the ratios conjecture as presented in the last section with one $L$-function in the numerator and one $L$-function in the denominator we arrive at the following particular conjecture.

\begin{conjecture}\label{R'_P}
	With $-\frac{1}{4}<\mathfrak{R}(\alpha)<\frac{1}{4}, \frac{1}{\log|P|}\ll\mathfrak{R}(\gamma)<\frac{1}{4}$ and $\mathfrak{I}(\alpha),\mathfrak{I}(\gamma)\ll_\epsilon|P|^{1-\epsilon}$ for every $\epsilon>0,$ we have
		
	\begin{equation}
	\begin{split}
	R_P(\alpha;\gamma)= & \sum_{P\in\p}\frac{L(\frac{1}{2}+\alpha,\x_P)}{L(\frac{1}{2}+\gamma,\x_P)}\\
	= & \sum_{P\in\p} \Bigg( \frac{\z(1+2\alpha)}{\z(1+\alpha+\gamma)} + |P|^{-\alpha} \xx\left(\tfrac{1}{2}+\alpha\right)\\
	& \times \frac{\z(1-2\alpha)}{\z(1-\alpha+\gamma)}\Bigg)+o\left(|P|\right).
	\end{split}
	\end{equation}
\end{conjecture}

To compute the one-level density we need to have a formula for 

\begin{equation}
\begin{split}
\sum_{P\in\p} \frac{L'(\frac{1}{2}+r,\x_P)}{L(\frac{1}{2}+r,\x_P)} =&\frac{d}{d\alpha} R_\mathcal{P}(\alpha;\gamma)\Big|_{\alpha=\gamma=r}.
\end{split}
\end{equation}
A direct calculation gives 

\begin{equation}
\begin{split}
\frac{d}{d\alpha}\left(\frac{\z(1+2\alpha)}{\z(1+\alpha+\gamma)}\right)\Bigg|_{\alpha=\gamma=r}=\frac{\z'(1+2r)}{\z(1+2r)}
\end{split}
\end{equation}
and that

\begin{equation}
\begin{split}
\frac{d}{d\alpha}\Big(|P|^{-\alpha} \xx(\tfrac{1}{2}+\alpha)&\frac{\z(1-2\alpha)}{\z(1-\alpha+\gamma)}\Big)\Bigg|_{\alpha=\gamma=r}\\
= & - \left(\log q\right) |P|^{-r} \xx\left(\tfrac{1}{2}+r\right) \z(1-2r).\\
\end{split}
\end{equation}
\newline

Therefore, the ratios conjecture implies that the following result holds.

\begin{theorem}\label{L'/L}
	Assuming Conjecture \ref{R'_P}, $\frac{1}{\log|P|}\ll\mathfrak{R}(r)<\frac{1}{4}$ and $\mathfrak{I}(r)\ll_\epsilon|P|^{1-\epsilon}$ for every $\epsilon>0,$ we have
	
	\begin{equation}
	\begin{split}
	\sum_{P\in\p} &\frac{L'(\frac{1}{2}+r,\x_P)}{L(\frac{1}{2}+r,\x_P)}\\
	= &  \sum_{P\in\p}  \Big( \frac{\z'(1+2r)}{\z(1+2r)}- \left(\log q\right) |P|^{-r} \xx\left(\tfrac{1}{2}+r\right) \\
	& \times \z(1-2r)\Big) +o\left(|P|\right).
	\end{split}
	\end{equation}
\end{theorem}

We have available all the necessary machinery to derive the formula for the one-level density for the zeros of Dirichlet $L$-functions associated to quadratic characters $\x_P$ with $P\in\p$, complete with lower order terms.
\newline

Let $\gamma_P$ be the ordinate of a generic zero of $L(s,\x_P)$ on the half-line. Since $L(s,\x_P)$ is a function of $u=q^{-s}$ and periodic with period $2\pi i/\log q$ we can restrict our analysis of the zeros for the range $-\pi i/\log q \le \mathfrak{I}(s) \le \pi i/\log q.$ Consider the one-level density 

\begin{equation}
S_1(f):=\sum_{P\in\p} \sum_{\gamma_P} f(\gamma_P),
\end{equation}
where $f$ is an even $2\pi/\log q$-periodic test functions and holomorphic. 
\newline

Using Cauchy's Theorem we have

\begin{equation}
S_1(f)=\sum_{P\in \p} \frac{1}{2\pi i} \left(\int_{(c)}-\int_{(1-c)}\right) \frac{L'(s,\x_P)}{L(s,\x_P)} f\left(-i\left(s-1/2\right)\right)ds,
\end{equation}
where $(c)$ is the vertical line from $c-\pi i/\log q$ to $c+\pi i/\log q$ and $1/2+1/\log|P|<c<3/4$. For the integral on the $(c)$-line, we make the following variable change, letting $s\to c+it$, so 

\begin{equation}
\begin{split}
\frac{1}{2\pi}\int_{-\pi/\log q}^{\pi/\log q}f(-i(it+c-1/2)) \sum_{P\in \p} \frac{L'(c+it,\x_P)}{L(c+it,\x_P)}dt.
\end{split}
\end{equation}
\newline

Since the integrand is regular at $t=0$, we move the path of the integration to $c=1/2$, and replace the sum over $P$ by Theorem \ref{L'/L} to obtain

\begin{equation}
\begin{split}
\frac{1}{2\pi}\int_{-\pi/\log q}^{\pi/\log q} & f(t) \sum_{P\in\p}  \Bigg( \frac{\z'(1+2it)}{\z(1+2it)}\\
& - \left(\log q\right) |P|^{-it} \xx\left(\frac{1}{2}+r\right) \z(1-2it)\Bigg)dt + o\left(|P|\right).
\end{split}
\end{equation}
\newline

The functional equation (\ref{functional equation}) implies that

\begin{equation}\label{l'/l}
\begin{split}
\frac{L'(1-s,\x_P)}{L(1-s,\x_P)}= \frac{\xx_P'(s)}{\xx_P(s)}-\frac{L'(s,\x_P)}{L(s,\x_P)},
\end{split}
\end{equation}
with 

\begin{equation}
\frac{\xx_P'(s)}{\xx_P(s)}= - \log|P|+\frac{\xx'(s)}{\xx(s)}.
\end{equation}
\newline

For the integral on the $(1-c)$-line, we change variables, letting $s\to 1-s$, then use (\ref{l'/l}) and with the similar calculations as for the integral on the $(c)$-line we obtain the following theorem.

\begin{theorem}\label{1-level}
	Assuming the ratios Conjecture \ref{R'_P}, we have that
	
	\begin{equation}
	\begin{split}
	S_1(f)=& \sum_{P\in\p} \sum_{\gamma_P} f(\gamma_P)\\
	=&  \frac{1}{2\pi} \int_{-\pi/\log q}^{\pi/\log q} f(t) \sum_{P\in \p} \Bigg( \log|P|+\frac{\xx'(\frac{1}{2}-it)}{\xx(\frac{1}{2}-it)} \\
	& + 2  \Bigg(  \frac{\z'(1+2it)}{\z(1+2it)}- \left(\log q\right) |P|^{-it} \xx\left(\tfrac{1}{2}+r\right) \z(1-2it)\Bigg) \Bigg)dt \\
	& + o\left(|P|\right),
	\end{split}
	\end{equation}
	where $\gamma_P$ is the ordinate of a generic zero of $L(s,\x_P)$ and $f$ is an even and periodic sutable test function.
\end{theorem}

\subsection{The Scaled One-Level Density}$\text{\color{white}kgm}$\\

Defining 

\begin{equation}
f(t)= h\left(\frac{t(2g\log q)}{2\pi}\right)
\end{equation}
and scaling the variable $t$ from Theorem \ref{1-level} as 

\begin{equation}
\tau=\frac{t(2g\log q)}{2\pi},
\end{equation}
we have that

\begin{equation}\label{s_1}
\begin{split}
\sum_{P\in\p} & \sum_{\gamma_P} f\left(\gamma_P \frac{2g\log q}{2\pi}\right)\\
=&  \frac{1}{2g \log q} \int_{-g}^{g} h(\tau) \sum_{P\in \p} \Bigg( \log|P|+\frac{\xx'\left(\frac{1}{2}-\frac{2\pi i \tau}{2g\log q}\right)}{\xx\left(\frac{1}{2}-\frac{2\pi i \tau}{2g\log q}\right)} \\
& + 2  \Bigg(  \frac{\z'\left(1+\frac{4\pi i\tau}{2g\log q}\right)}{\z\left(1+\frac{4\pi i\tau}{2g\log q}\right)}- \left(\log q\right) e^{(-2\pi i \tau/2g \log q)\log|P|} \xx\left(\tfrac{1}{2}+\tfrac{2\pi i \tau}{2g\log q}\right)\\
& \times  \z\left(1-\frac{4\pi i\tau}{2g\log q}\right)\Bigg) \Bigg)d\tau + o\left(|P|\right).
\end{split}
\end{equation}
\newline

Writing 

\begin{equation}
\begin{split}
\z(1+s)= \frac{1}{s \log q} + \frac{1}{2} + \frac{1}{12} (\log q) s + O(s^2),
\end{split}
\end{equation}
and

\begin{equation}
\begin{split}
\frac{\z'(1+s)}{\z(1+s)}= -s^{-1} + \frac{1}{2}\log q - \frac{1}{12} (\log q)^2 s + O(s^3),
\end{split}
\end{equation}
we have

\begin{equation}
\begin{split}
&\sum_{P\in\p}  \sum_{\gamma_P} f\left(\gamma_P \frac{2g\log q}{2\pi}\right)\\
&=  \frac{1}{2g \log q} \int_{-g}^{g} h(\tau) \sum_{P\in \p} \Bigg( \log|P|+\frac{\xx'\left(\frac{1}{2}-\frac{2\pi i \tau}{2g\log q}\right)}{\xx\left(\frac{1}{2}-\frac{2\pi i \tau}{2g\log q}\right)} \\
& + 2  \Bigg(  -\frac{2g\log q}{4\pi i\tau} + \frac{1}{2}\log q - \frac{1}{12} (\log q) \frac{4\pi i\tau}{2g}  - \left(\log q\right) e^{(-2\pi i \tau/2g \log q)\log|P|} \\
& \times \xx\left(\tfrac{1}{2}+\tfrac{2\pi i \tau}{2g\log q}\right) \left(-\frac{2g}{4\pi i \tau} + \frac{1}{2} - \frac{1}{12} \frac{4\pi i\tau}{2g}\right)\Bigg) \Bigg)d\tau + o\left(|P|\right).
\end{split}
\end{equation}
then, for $g$ large, only the term $\log|P|$, the $\zeta_{A}^{'}/\zeta_{A}$ and the final term in the integral contribute, yielding the asymptotic

\begin{equation}
\begin{split}
\sum_{P\in\p} & \sum_{\gamma_P} f\left(\gamma_P \frac{2g\log q}{2\pi}\right)\\
\sim & \, \frac{1}{2g \log q} \int_{-\infty}^{\infty} h(\tau) \Bigg(\left(\#\p\right) \log|P|\\
& \color{white} \text{gdkjgj j} \color{black} - \left(\#\p\right) \frac{2g \log q}{2\pi i \tau} +  \left(\#\p\right) e^{-2\pi i \tau} \frac{2g \log q}{2\pi i \tau}\Bigg) d\tau.
\end{split}
\end{equation}
\newline

However, since $h$ is an even function, we can drop out the mid term and the last term can be duplicated with a change of sign of $\tau,$ leaving 

\begin{equation}
\begin{split}
\lim_{g\to\infty} & \frac{1}{\#\p} \sum_{P\in\p}  \sum_{\gamma_P} f\left(\gamma_P \frac{2g\log q}{2\pi}\right)\\
= & \,  \int_{-\infty}^{\infty} h(\tau) \Bigg(1 + e^{-2\pi i \tau} \frac{1}{2\pi i \tau}+ e^{2\pi i \tau} \frac{1}{-2\pi i \tau}\Bigg) d\tau\\
= & \,  \int_{-\infty}^{\infty} h(\tau) \Bigg(1 + \frac{1}{2\pi  \tau}\Big(\left(\cos(2\pi \tau) -\sin(2\pi \tau)\right) - \left(\cos(2\pi\tau)-\sin(2\pi\tau)\right) \Big)\Bigg) d\tau\\
= & \,  \int_{-\infty}^{\infty} h(\tau) \Bigg(1 + \frac{1}{2\pi  \tau}\Big(-2\sin(2\pi\tau) \Big)\Bigg) d\tau\\
= & \,  \int_{-\infty}^{\infty} h(\tau) \Bigg(1 - \frac{\sin(2\pi\tau)}{\pi  \tau}\Bigg) d\tau.
\end{split}
\end{equation}


\vspace{0.5cm}


\noindent \textit{Acknowledgement.}  
The first author is grateful to the Leverhulme Trust (RPG-2017-320) for the support through the research project grant ``Moments of $L$-functions in Function Fields and Random Matrix Theory". The research of the second author is supported by Basic Science Research Program through the National Research Foundation of Korea (NRF) 
funded by the Ministry of Education (2018015574). The third author is supported by a Ph.D. scholarship from the government of Kuwait.


\end{document}